\documentclass[10pt]{amsart}
\usepackage{amssymb}
\usepackage{bm}
\usepackage{amsfonts}
\usepackage{amssymb,amsfonts,amsmath,amsthm,graphicx}
\usepackage{comment}
\usepackage{slashbox,mathrsfs}
\usepackage[all,poly]{xy}
\usepackage{float}

\allowdisplaybreaks

\begin{document}
\theoremstyle{plain}
\newtheorem{thm}{Theorem}[section]
\newtheorem{theorem}[thm]{Theorem}
\newtheorem*{theorem2}{Theorem}
\newtheorem{lemma}[thm]{Lemma}
\newtheorem{corollary}[thm]{Corollary}
\newtheorem{corollary*}[thm]{Corollary*}
\newtheorem{proposition}[thm]{Proposition}
\newtheorem{proposition*}[thm]{Proposition*}
\newtheorem{conjecture}[thm]{Conjecture}
\theoremstyle{definition}
\newtheorem{construction}[thm]{Construction}
\newtheorem{notations}[thm]{Notations}
\newtheorem{question}[thm]{Question}
\newtheorem{problem}[thm]{Problem}
\newtheorem{remark}[thm]{Remark}
\newtheorem{remarks}[thm]{Remarks}
\newtheorem{definition}[thm]{Definition}
\newtheorem{claim}[thm]{Claim}
\newtheorem{assumption}[thm]{Assumption}
\newtheorem{assumptions}[thm]{Assumptions}
\newtheorem{properties}[thm]{Properties}
\newtheorem{example}[thm]{Example}
\newtheorem{comments}[thm]{Comments}
\newtheorem{blank}[thm]{}
\newtheorem{observation}[thm]{Observation}
\newtheorem{defn-thm}[thm]{Definition-Theorem}

\def\supp{\operatorname{supp}}


\title[An explicit formula for the Berezin star product]{An explicit formula for the Berezin star product}

\author{Hao Xu}
        \address{Department of Mathematics, Harvard University, Cambridge, MA 02138, USA}
        \email{haoxu@math.harvard.edu}

        \begin{abstract}
        We
        prove an explicit formula of the Berezin star product on K\"ahler manifolds. The formula is expressed as a summation
        over certain strongly connected digraphs. The proof relies on a combinatorial interpretation of Engli\v{s}' work on the asymptotic
        expansion of the Laplace integral.
        \end{abstract}


    \maketitle

\section{Introduction}

Bayen, Flato, Fronsdal, Lichnerowicz and Sternheimer \cite{BFF,
BFF2} introduced quantization as a deformation of the usual
commutative product into a noncommutative associative
$\star$-product. The existence of $\star$-products on symplectic
manifolds was solved by De Wilde and Lecomte \cite{WL}, Omori, Maeda
and Yoshioka \cite{OMY, OMY2}, and Fedosov \cite{Fed}. The most
spectacular result on the existence of $\star$-products was
Kontsevich's proof \cite{Kon} of a universal formula that gives a
$\star$-product on any Poisson manifold.

Let $(M,g)$ be a K\"ahler manifold of dimension $n$. On a coordinate
chart $\Omega$, the K\"ahler form is given by
$$\omega_g=\frac{\sqrt{-1}}{2\pi}\sum_{i,j=1}^n
g_{i\overline{j}}dz_i\wedge dz_{\overline{j}}.$$ If $\Omega$ is
contractible, there exists a K\"{a}hler potential $\Phi$ satisfying
$$\partial\overline\partial\Phi=\sum_{i,j=1}^n
g_{i\overline{j}}dz_i\wedge dz_{\overline{j}}.$$

Recall that around each point $x\in M$, there exists a normal
coordinate system such that
\begin{equation} \label{eqnormal}
g_{i\bar j}(x)=\delta_{ij}, \qquad g_{i\bar j k_1\dots
k_r}(x)=g_{i\bar j \bar l_1\dots\bar l_r}(x)=0
\end{equation}
for all $r\leq N\in \mathbb N$, where $N$ can be chosen arbitrary
large and $g_{i\bar j k_1\dots k_r}=\partial_{k_1\dots k_r}g_{i\bar
j}$.

The Poisson bracket of the functions $f_1,f_2\in C^\infty(M)$ can be
expressed locally as
\begin{equation}
\{f_1,f_2\}=i g^{k\bar l}\left(\frac{\partial f_1}{\partial
z^k}\frac{\partial f_2}{\partial\bar z^l}-\frac{\partial
f_2}{\partial z^k}\frac{\partial f_1}{\partial\bar z^l}\right).
\end{equation}

Let $C^\infty(M)[[h]]$ denote the algebra of formal power series in
$h$ over $C^\infty(M)$. A star product is an associative $\mathbb
\mathbb C[[h]]$-bilinear product $\star$ such that $\forall
f_1,f_2\in C^\infty(M)$,
\begin{equation} \label{eqberdef}
f_1\star f_2=\sum_{j=0}^\infty h^j C_j(f_1,f_2),
\end{equation}
where the $\mathbb C$-bilinear operators $C_j$ satisfy
\begin{equation}
C_0(f_1,f_2)=f_1 f_2,\qquad C_1(f_1,f_2)-C_1(f_2,f_1)=i\{f_1,f_2\}.
\end{equation}

There are constructions of $\star$-products on restricted types of
K\"ahler manifolds by Berezin \cite{Ber, Ber2}, Moreno and
Ortega-Navarro \cite{Mor, MO}, and Cahen, Gutt and Rawnsley
\cite{CGR, CGR2, CGR3}. The existence and classification of
deformation quantizations with ``separation of variables'' for all
K\"ahler manifolds was shown by Karabegov \cite{Kar}. There are
alternative constructions by Bordemann and Waldmann \cite{BW},
Reshetikhin and Takhtajan \cite{RT}, Schlichenmaier \cite{Sch},
Engli\v{s} \cite{Eng3} and Gammelgaard \cite{Gam}. See also the
recent preprints of Karabegov \cite{Kar4, Kar5}.

Let $\Phi(x,y)$ be an almost analytic extension of $\Phi(x)$ to a
neighborhood of the diagonal, i.e. $\bar\partial_x\Phi$ and
$\partial_y\Phi$ vanish to infinite order for $x=y$ (cf. \cite{BS}).
We can assume $\overline{\Phi(x,y)}=\Phi(y,x)$. Consider the real
valued function
$$D(x, y) = \Phi(x, x) + \Phi(y, y)-\Phi(x, y)-\Phi(y, x),$$
which is called the Calabi diastatic function \cite{Cal}. It is
easily seen that in a sufficiently small neighborhood of the
diagonal, $D(x, y)\geq0$ and $D(x, y)=0$ if and only $x=y$.

For $\alpha > 0$, consider the weighted Bergman space of all
holomorphic function on $\Omega$ square-integrable with respect to
the measure $e^{-\alpha \Phi}\frac{w_g^n}{n!}$.

We denote by $K_\alpha(x,y)$ the reproducing kernel. Locally, it is
often the case that $K_\alpha(x,y)$ has an asymptotic expansion in a
small neighborhood of the diagonal  (see \cite{Ber, Eng, KS}) for
$\alpha\rightarrow\infty$,
\begin{equation}\label{eqb1}
K_\alpha (x,y)\sim e^{\alpha \Phi (x,y)} \sum^\infty_{k=0} B_k (x,y)
\alpha^{n-k}.
\end{equation}
To ensure the convergence of this expansion, we could take the
sufficiently small neighborhood $\Omega$ to be a strongly pseudoconvex domain
with real analytic boundary (cf. \cite{Eng, Eng2}). We will implicitly take such
$\Omega$ throughout the paper, which is needed to guarantee the
convergence of other asymptotic expansions such as \eqref{eqber3}.
 The Bergman
kernels $B_k(x,x)$ on the diagonal turn out to be linear
combinations with universal coefficients of Weyl invariants (see
Section \ref{graph} for the definition) independent of $\Omega$. For
discussions on the convergence in the compact K\"ahler case, see
\cite{KS}.

The Bergman kernel $B_k$ in the setting when $\Omega$ is a compact
K\"ahler manifold was also much studied (cf. \cite{Cat, Tia, Zel,
Lu, Loi}). Dai, Liu and Ma \cite{DLM} proved the most general
version for the asymptotic expansion of Bergman kernels on orbifolds
and symplectic manifolds.

The {\it Berezin transform} is the integral operator
\begin{equation} \label{eqber}
I_\alpha f(x)=\int_\Omega
f(y)\frac{|K_\alpha(x,y)|^2}{K_\alpha(x,x)}e^{-\alpha
\Phi(y)}\frac{w_g^n(y)}{n!}.
\end{equation}
At any point for which $K_\alpha (x,x)$ invertible, the integral
converges for each bounded measurable function $f$ on $\Omega$. Note
that \eqref{eqb1} implies that for any $x$, $K_\alpha (x,x)\neq 0$
if $\alpha$ is large enough.

The Berezin transform has an asymptotic expansion for
$\alpha\rightarrow\infty$ (cf. \cite{Eng, KS}),
\begin{equation} \label{eqber3}
I_\alpha f(x)=\sum^\infty_{k=0} Q_k f(x)\alpha^{-k},
\end{equation}
where $Q_k$ are linear differential operators.

The Berezin star product was introduced by Berezin \cite{Ber}
through symbol calculus for linear operators on weighted Bergman
spaces (cf. \cite{Eng3}). As noted by Karabegov \cite{Kar2}, the
Berezin star product is related to the asymptotic expansion of the
Berezin transform in a nice way (cf. \cite{Eng4}). Denote by
$c_{j\alpha\beta}\partial^{\alpha}\bar\partial^{\beta}$ the
coefficients of $Q_j$,
\begin{equation}
Q_j f=\sum_{\alpha,\beta\text{
multiindices}}c_{j\alpha\beta}\partial^{\alpha}\bar\partial^{\beta}f.
\end{equation}
Then the coefficients of the Berezin star product are given by
bilinear differential operators
\begin{equation} \label{eqkar}
C_j(f_1,f_2):=\sum_{\alpha,\beta}
c_{j\alpha\beta}(\bar\partial^{\beta}f_1)(\partial^{\alpha} f_2).
\end{equation}

The Berezin star product is equivalent to the Berezin-Toeplitz star
product $\star_{BT}$ via the Berezin transform (cf. \cite{KS})
\begin{equation} \label{eqbt}
f_1\star_{BT} f_2=I^{-1}(I f_1\star I f_2),
\end{equation}
where $I:=I_{1/h}$ is obtained from substituting $\alpha$ by $1/h$
in $I_{\alpha}$.

Recall that the Toeplitz operator $T^{(m)}_f$ for $f\in C^\infty(M)$
is defined to be
\begin{equation}
T^{(m)}_f:=\Pi^{(m)}(f\cdot):\quad H^0(M,L^m)\rightarrow H^0(M,L^m),
\end{equation}
where $\Pi^{(m)}:L^2(M,L^m)\rightarrow H^0(M,L^m)$ is the
projection.

It was proved by Schlichenmaier \cite{Sch} that the Berezin-Toeplitz
star product \eqref{eqbt} is the unique star product
\begin{equation}
f_1\star_{BT} f_2:=\sum_{j=0}^\infty h^j C^{BT}_j(f_1,f_2),
\end{equation}
such that the following asymptotic expansion holds
\begin{equation} \label{eqbt2}
T_{f_1}^{(m)}T_{f_2}^{(m)} \sim \sum_{j=0}^\infty m^{-j}
T_{C^{BT}_j(f_1,f_2)}^{(m)},\qquad m\rightarrow\infty.
\end{equation}

The Berezin transform was introduced by Berezin \cite{Ber2} for
symmetric domains in $\mathbb C^n$. Among the pioneers in extending
Berezin's results are Unterberger and Upmeier \cite{UU}, Engli\v{s}
\cite{Eng0}, and Engli\v{s} and Peetre \cite{EP}. The coefficients
$Q_i,\,i\leq 3$ have been obtained by Engli\v{s} \cite{Eng}.
Karabegov and Schlichenmaier \cite{KS} proved the asymptotic
expansion of the Berezin transform for compact K\"ahler manifolds
and showed that the Berezin-Toeplitz deformation quantization has
the property of separation of variables. Ma-Marinescu
\cite[Ch.7]{MM3, MM2} developed the theory of Toeplitz operators on
symplectic manifolds in the presence of a twisting vector bundle and
showed that the calculation of the coefficients of the
Berezin-Toeplitz star product is local. Moreover, in the K\"ahler
case, Ma-Marinescu \cite{MM} calculated the first three terms of the
expansion of the kernels of Toeplitz operators and the
Berezin-Toeplitz star product with a twisting vector bundle.
(cf. also \cite{BMS, BG, Cha, Eng3, Eng4, Fin, Kar3} and especially
the nice survey by Schlichenmaier \cite{Sch2} for related works on
Berezin-Toeplitz quantization).

Using the closed graph-theoretic formula \eqref{eqq}, we computed
$Q_k,\,k\leq 4$ in Example \ref{computeq} and the appendix. We also
computed the first four coefficients of the Berezin-Toeplitz star
product in Example \ref{computebt}.

As discussed in Section \ref{graph}, we can write $Q_k f$ as a sum
over graphs.
\begin{equation}
Q_k f=\sum_{\Gamma\in  \dot{\mathcal G}(k)}Q_{\Gamma}\Gamma,\qquad
k\geq0.
\end{equation}
The closed formula \eqref{eqq} of $Q_{\Gamma}$ leads to the main
result of this paper, which is the following explicit formula for
the Berezin star product.

\begin{theorem} \label{main} Let $M$ be a K\"ahler manifold.
Fix a normal coordinate system around $x\in M$, the following
equation for the Berezin $\star$-product holds at $x$.
\begin{equation} \label{eqmain}
f_1\star f_2(x)=\sum_{\Gamma=(V\cup\{f\},E)\in\dot{\mathcal
G}_{scon}}\frac{\det(A(\Gamma_-)-I)}{|\dot{\rm Aut}(\Gamma)|}
h^{|E|-|V|} D_{\Gamma}(f_1,f_2)\Big|_x,
\end{equation}
where $\Gamma$ runs over the set of strongly connected pointed
stable graphs; $\Gamma_-$ is the subgraph of $\Gamma$ obtained by
removing the distinguished vertex $f$ from $\Gamma$ and $A(\Gamma_-)$ is
its adjacency matrix.  $D_{\Gamma}(f_1,f_2)$ is the partition
function of $\Gamma$ (see Definition \ref{partition} in Section
\ref{graph}). Moreover, we can convert the right-hand side to the
invariant tensor expression easily.
\end{theorem}

The expression of the Berezin star product as a summation over
graphs is not surprising as there are remarkable pioneering works by
Kontsevich \cite{Kon}, and Reshetikhin and Takhtajan \cite{RT}.
However, the formula \eqref{eqmain} may be the simplest concerning
the star-products.

Kontsevich's universal formula of a star product on Poisson
manifolds was written as a summation over labeled directed graphs
with two distinguished vertices and the coefficients are certain
integrals over configuration spaces. The Feynmann graph formula of
Reshetikhin and Takhtajan is for the non-normalized Berezin star
product ($1$ is not the unit in the product) on arbitrary K\"ahler
manifolds. In an attempt to normalize Reshetikhin and Takhtajan's
formula, Gammelgaard \cite{Gam} obtained a universal formula for any
star product with separation of variables corresponding to a given
classifying Karabegov form. His formula is expressed as a summation
over weighted acyclic graphs. As pointed out by a referee,
Gammelgaad's formula does not allow to express the Berezin star
product directly since the Karabegov form of the Berezin star
product is not explicitly known in general. In
the quantizable compact K\"ahler case, the Karabegov form of the Berezin
star product was given in terms of the Bergman kernel evaluated along the diagonal (cf. \cite[\S 7]{Sch2})
and the Karabegov form of the Berezin-Toeplitz star product was also identified in \cite{KS}.

Although Berezin transform \eqref{eqber} can be defined only under very restrictive conditions, the coefficients
$Q_k$ of its asymptotic expansion, as long as it exists, are universal polynomials of curvature tensor independent of
the underlying K\"ahler metric.
As a result, the formula \eqref{eqmain} defines a universal (associative) star product with separation of variables on
an arbitray K\"ahler manifold. A proof will be given at the end of \S \ref{berezin}.

\

\noindent{\bf Acknowledgements} The author is grateful to Professor
Kefeng Liu, whose help and guidance saved my math career. The author
thanks the referees for very helpful comments and suggestions.

\vskip 30pt
\section{Weyl invariants and graphs} \label{graph}

A {\it digraph} (directed graph) $G=(V,E)$ is defined to be a finite
set $V$ (whose elements are called vertices) and a multiset $E$ of
ordered pairs of vertices, called directed edges. Throughout the
paper, we allow a digraph to have loops and multi-edges. The
indegree and outdegree of a vertex $v$ are denoted by $\deg^-(v)$
and $\deg^+(v)$ respectively. The adjacency matrix $A=A(G)$ of a
digraph $G$ with $n$ vertices is a square matrix of order $n$ whose
entry $A_{ij}$ is the number of directed edges from vertex $i$ to
vertex $j$.

A digraph $G$ is called {\it connected} if the underlying undirected
graph is connected, and {\it strongly connected} if there is a
directed path from each vertex in $G$ to every other vertex. For a
digraph $G=(V,E)$, we can partition $V$ into strongly connected
components, namely the maximal strongly connected subgraphs of $G$.
Among these strongly connected components, we have at least one {\it
sink} (a component without outgoing edges) and one {\it source} (a
component without incoming edges).

A {\it full subdigraph} $H$ of a digraph $G$ is a digraph whose
vertex set is a subset of the vertex set of $G$, the set of edges
connecting any two vertices $v_1$ and $v_2$ in $H$ equals the set of
edges connecting $v_1$ and $v_2$ in $G$.

Given two subgraphs $G_1$ and $G_2$ of $G$, we denote by $G_1+G_2$
the full subgraph of $G$ whose vertices are the vertices of $G_1$
and $G_2$.

Recall the definition of stable and semistable graphs in the
previous paper \cite{Xu}. These graphs were used to represent Weyl
invariants, which include the coefficients of the asymptotic
expansion of the Bergman kernel.
\begin{definition}
We call a vertex $v$ of a digraph $G$ semistable if we have
$$\deg^-(v)\geq1,\ \deg^+(v)\geq1,\ \deg^-(v)+\deg^+(v)\geq3.$$
$G$ is called {\it semistable} if each vertex of $G$ is semistable.
We call $v$ stable if $\deg^-(v)\geq2,\ \deg^+(v)\geq2$. A digraph
$G$ is {\it stable} if each vertex of $G$ is stable. The set of
semistable and stable graphs will be denoted by $\mathcal G^{ss}$
and $\mathcal G$ respectively. For any $G=(V,E)\in \mathcal G^{ss}$,
its weight $\omega(G)$ is defined to be the integer $|E|-|V|$. We
denote by $\mathcal G^{ss}(k)$ and $\mathcal G(k)$ respectively the
set of semistable and stable digraphs with weight $k$. Let $\mathcal
G_{con}(k)$ and $\mathcal G_{scon}(k)$ respectively be the set of
connected and strongly connected graphs in $\mathcal G(k)$.
\end{definition}

We need to make slight extension of these concepts, in order to
write the Berezin star product as a summation over graphs.

\begin{definition}
A {\it (one-)pointed semistable (stable) graph} $\Gamma=(V\cup\{f\},E)$ is
defined to be a digraph with a distinguished vertex labeled by $f$, such
that each ordinary vertex $v\in V$ is semistable (stable). We denote
by $\dot{\rm Aut}(\Gamma)$ all automorphisms of the pointed graph
$\Gamma$ fixing the vertex $f$. Let $\Gamma_-$ denote the subgraph
of $\Gamma$ obtained by removing the distinguished vertex $f$ from
$\Gamma$ and $A(\Gamma_-)$ its adjacency matrix.

The set of pointed semistable and stable graphs will be denoted by
$\dot{\mathcal G}^{ss}$ and $\dot{\mathcal G}$ respectively. For any
$\Gamma\in \dot{\mathcal G}^{ss}$, its weight $w(\Gamma)$ is defined
to be the integer $|E|-|V|$. We denote by $\dot{\mathcal G}^{ss}(k)$
and $\dot{\mathcal G}(k)$ respectively the set of pointed semistable
and pointed stable digraphs with weight $k$. Let $\dot{\mathcal
G}_{con}(k)$ ($\dot{\mathcal
G}_{con}^{ss}(k)$) and $\dot{\mathcal G}_{scon}(k)$ ($\dot{\mathcal G}_{scon}^{ss}(k)$) respectively be the
set of connected and strongly connected graphs in $\dot{\mathcal
G}(k)$ ($\dot{\mathcal
G}^{ss}(k)$). We also define a special set of graphs:
$$\dot\Lambda(k)=\{\Gamma\in\dot{\mathcal G}_{scon}(k)\mid 1 \text{ is not an
eigenvalue of } A(\Gamma_-)\}.$$ The cardinality of this set is the
number of terms in $Q_k$ by Theorem \ref{engq}.
We have computed the cardinalities of these sets when $k\leq5$ in
Table \ref{tb1}.
\end{definition}
\begin{table}[h]
\caption{Numbers of pointed stable graphs} \label{tb1}
\begin{tabular}{|c||c|c|c|c|c|c|}
\hline
    $k$                                            &$0$&$1$&$2$&$3$ &$4$  &$5$
\\\hline           $|\dot{\mathcal G}(k)|$         &$1$&$2$&$9$&$46$&$314$&$2638$
\\\hline       $|\dot{\mathcal G}_{con}(k)|$       &$1$&$1$&$4$&$23$&$178$&$1637$
\\\hline       $|\dot{\mathcal G}_{scon}(k)|$      &$1$&$1$&$2$&$9$ &$61$ &$538$
\\\hline       $|\dot\Lambda(k)|$                  &$1$&$1$&$1$&$5$ &$36$ &$331$
\\\hline
\end{tabular}
\end{table}

\begin{remark} \label{rm1}
Note that $\mathcal G^{ss}$ ($\mathcal G$) may be regarded as a
subset of $\dot{\mathcal G}^{ss}$ ($\dot{\mathcal G}$) where the
distinguished vertex $f$ is an isolated vertex without loops.
\end{remark}

\begin{remark} \label{rm4}
For later use, we may extend the above definition to semistable (stable) graphs with $m$ distinguished vertices
$\Gamma=(V\cup\{f_1,\dots,f_m\},E)$. We denote by $\tilde{\mathcal G}$ ($\tilde{\mathcal G}^{ss}$) the set of
semistable (stable) graphs with any number of distinguished vertices. The automorphism of a graph
$\Gamma\in\tilde{\mathcal G}$ ($\Gamma\in\tilde{\mathcal G}^{ss}$) is
always assumed to fix each distinguished vertex and its automorphism group is simply denoted by ${\rm Aut}(\Gamma)$.
\end{remark}

The concept of {\it Weyl invariants} was introduced by Fefferman
\cite{Fef2}. We slightly extend the definition to allow additional
functions. Consider the tensor products of covariant derivatives of
the curvature tensor $R_{i\bar j k \bar l; p\cdots \bar q}$ and a
function $f$, e.g.
$$R_{ijk \bar l;p \bar q}\otimes\cdots \otimes  R_{a \bar b c \bar d; \bar
e}\otimes f_{;r\bar st}.$$ The Weyl invariants are constructed by
first pairing up the unbarred indices to barred indices and then
contracting all paired indices.

As remarked in \cite{Xu}, we can write Weyl invariants as
polynomials of $g_{i\bar j \alpha}$ and $f$ and their derivatives.
The advantage is that we do not need to deal with the problem of
exchanging indices, thus we can canonically represent a Weyl
invariant as a sum over pointed stable graphs, from which we can
easily recover its curvature-tensor expression.

We represent a digraph as a weighted digraph.  The weight of a
directed edge is the number of multi-edges. The number attached to a
vertex denotes the number of its self-loops. A vertex without loops
will be denoted by a small hollow circle $\circ$. The distinguished
vertex $f$ is denoted by a solid circle $\bullet$. We will drop it
if $f$ is an isolated vertex without loops.

The Weyl invariant $g_{i \bar i k \bar l p} g_{j \bar j l \bar k
\bar q} f_{q\bar p}$ is depicted in Figure \ref{figweyl} knowing
that $(i,\bar i),\,(j, \bar j)$ etc. are paired indices to be
contracted.
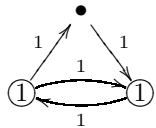
\begin{figure}[h]
$\xymatrix@C=4mm@R=7mm{
                & \bullet  \ar[dr]^{1}             \\
         *+[o][F-]{1} \ar[ur]^{1} \ar@/^0.4pc/[rr]^{1} & &  *+[o][F-]{1}  \ar@/^0.4pc/[ll]^{1}
         }
$ \caption{The associated graph of $g_{i \bar i k \bar l p} g_{j
\bar j l \bar k \bar q} f_{q\bar p}$} \label{figweyl}
\end{figure}

Below, we may use the same notation to denote a graph and its
associated Weyl invariant.

Therefore we can write $R_k f(x)$ in \eqref{eqeng} and $Q_k f(x)$ in
\eqref{eqber3} as a weighted sum of pointed semistable graphs in
$\dot{\mathcal G}^{ss}(k)$.
\begin{equation} \label{eqengrq}
R_k f=\sum_{\Gamma\in  \dot{\mathcal G}^{ss}(k)}R_{\Gamma}\Gamma,
\qquad Q_k f=\sum_{\Gamma\in  \dot{\mathcal
G}^{ss}(k)}Q_{\Gamma}\Gamma.
\end{equation}

Note that at the center of the normal coordinate, terms of
non-stable graphs vanish. Moreover, to recover
$R_{\Gamma},Q_{\Gamma}$ for all semistable pointed graphs
$\Gamma\in\dot{\mathcal G}^{ss}$, it is enough to know only their
values for stable pointed graphs $\Gamma\in\dot{\mathcal G}$. We
will prove closed formulas of $R_{\Gamma},Q_{\Gamma}$ in Section
\ref{berezin}

\begin{definition} \label{partition}
Given $\Gamma\in\dot{\mathcal G}^{ss}$, the {\it partition function}
$D_{\Gamma}(f_1,f_2)$ is defined to be a Weyl invariant generated
from $\Gamma$ by replacing the vertex $f$ in with two vertices $f_1$
and $f_2$, such that all inward edges of $f$ are connected to $f_1$
and all outward edges of $f$ are connected to $f_2$.
\end{definition}

For example, if $\Gamma$ is the graph in Figure \ref{figweyl}, then
\begin{equation}
D_{\Gamma}(f_1,f_2)=g_{i \bar i k \bar l p} g_{j \bar j l \bar k
\bar q} \partial_{\bar p} f_1 \partial_q f_2.
\end{equation}

Let $\Gamma\in\dot{\mathcal G}^{ss}$ and $H\in \mathcal G^{ss}$, we
define $D_{\Gamma}(H)$ to be the Weyl invariant generated by
replacing $f$ in $\Gamma$ with $H$. For example, if $\Gamma$ is the
graph in Figure \ref{figweyl}, then
\begin{equation}
D_{\Gamma}(H)=g_{i \bar i k \bar l p} g_{j \bar j l \bar k \bar q}
\partial_{q\bar p} H.
\end{equation}
The derivatives of $H$ in the right-hand side may be expanded to get
a sum of semistable graphs.

\begin{remark}\label{rm3}
Given $\Gamma\in\dot{\mathcal G}$ and two Weyl invariants with
associated graphs $H_1,H_2\in\mathcal G^{ss}$, we
define $D_{\Gamma}(H_1,H_2)$ to be sum of stable graphs generated from $D_{\Gamma}(f_1,f_2)$
by replacing $f_1,f_2$ with $H_1,H_2$. We emphasize that in this definition, we discard all semistable but not stable graphs.
In particular, we only need to consider $\Gamma$ as acting on vertices of $H_1,H_2$ but not on edges, since in the latter
case, there will appear semistable but not stable vertices (cf. \cite[\S 2]{Xu}). Note that $D_{\Gamma}(H_1,H_2)$ may be
defined for $H_1,H_2\in\tilde{\mathcal G}^{ss}$ similarly.
\end{remark}

Let $\mathscr L$ be the set of digraphs consisting of a finite
number of vertex-disjoint simple cycles (i.e. simple polygons
without common vertex). The length of a simple cycle is defined to
be the number of its edges. For each graph $L\in \mathscr L$, we can
write $L$ as a finite increasing sequence of nonnegative integers
$[i_1,\dots, i_m]$, meaning $L$ consists of $m$ disjoint simple
cycles, whose lengths are specified by $i_1,\dots,i_m$. We define
the index of $L$ to be
\begin{equation}
i(L)=m+i_1+\dots+i_m.
\end{equation}
Note that $[0]$ is just a single vertex and $[1]$ is a vertex with a
self-loop. If $0\notin L$, then $L$ is usually called a linear
digraph. Recall that a {\it linear digraph} is a digraph in which
$\deg^+(v)=\deg^-(v)=1$ for each vertex $v$.

Given a set of indices ${\alpha_1,\dots,\alpha_r}$, denote by
$\mathscr L(\alpha_1,\dots,\alpha_r)$ the isomorphism classes of all
possible decorations of the vertices of $L\in \mathscr L$ with the
half-edges ${\alpha_1,\dots,\alpha_r}$ requiring each vertex to be
semistable. Two decorations of $L$ that differ by a graph
isomorphism are considered the same.

In \cite{Xu}, we proved the following lemma, which explains the
graphical properties of the partial derivatives of $\det g$.
\begin{lemma} \label{dettree}
We have
\begin{equation} \label{eqdettree}
\frac{1}{\det g}\partial_{\alpha_1}\dots\partial_{\alpha_r}\det g
=\sum_{\substack{L\in \mathscr L(\alpha_1,\dots,\alpha_r)\\0\notin
L}} (-1)^{i(L)} \cdot L.
\end{equation}
\end{lemma}

We need the following coefficient theorem (see Theorem 1.2 in
\cite{CDS}) from the spectral graph theory.
\begin{theorem} \label{cds}
Let $P_G (\lambda) = \lambda^n + c_1\lambda^{n-1} + \dots + c_n$
 be
the characteristic polynomial of a digraph $G$ with $n$ vertices.
Then for each $i = 1,\dots,n$,
$$c_i = \sum_{L\in\mathcal L_i} (-1)^{p(L)}, $$ where $\mathcal L_i$ is the set of all
linear directed subgraphs $L$ of $G$ with exactly $i$ vertices;
$p(L)$ denotes the number of components of $L$.
\end{theorem}

The following lemma and remark will be used in \S \ref{berezin}.
\begin{lemma} \label{graph2}
Let $\Gamma\in\dot{\mathcal G}$ and $H_1,H_2\in \mathcal G^{ss}$. (cf. Remark \ref{rm3}
for the definition of $D_{\Gamma}(H_1,H_2)$)
Then
\begin{equation}
\frac{1}{|\dot{\rm Aut}(\Gamma)| |{\rm
Aut}(H_1)| |{\rm
Aut}(H_2)|}D_{\Gamma}(H_1,H_2)=\sum_{G}\frac{1}{|{\rm Aut}(G)|}G,
\end{equation}
where $G$ in the summation runs over isomorphism classes of graphs appearing in the
expansion of $D_{\Gamma}(H_1,H_2)$.
\end{lemma}
\begin{proof}
Note that the group $\dot{\rm Aut}(\Gamma)\times {\rm Aut}(H_1)\times {\rm Aut}(H_2)$
has a natural action on the set
of all graphs in the expansion of $D_{\Gamma}(H_1,H_2)$. Then it is not
difficult to see that the set of orbits corresponds to isomorphism
classes of graphs and the isotropy group at a graph $G$ is just
${\rm Aut}(G)$. So we get the desired equation.
\end{proof}
\begin{remark} \label{rm2}
Let $\Gamma\in\dot{\mathcal G}$ be a (one-)pointed stable graph
and $H_1,H_2\in \tilde{\mathcal G}^{ss}$ semistable graphs with (any number of) distinguished vertices. (cf. Remark \ref{rm4}).
Then the above lemma still
holds without any change.
\begin{equation}
\frac{1}{|\dot{\rm Aut}(\Gamma)| |{\rm
Aut}(H_1)| |{\rm
Aut}(H_2)|}D_{\Gamma}(H_1,H_2)=\sum_{G}\frac{1}{|{\rm Aut}(G)|}G,
\end{equation}
where $G$ in the summation runs over isomorphism classes of graphs appearing in the
expansion of $D_{\Gamma}(H_1,H_2)$. Note that if $H_1,H_2$ has $m_1,m_2$ distinguished vertices
respectively, then $G$ has $m_1+m_2$ distinguished vertices.
\end{remark}

Finally, we record two identities that can convert covariant derivatives of curvature tensors to partial derivatives
of metrics and vice versa.
\begin{gather}\label{eqcur1}
R_{i\bar jk\bar l} =-g_{i\bar j k\bar l}+g^{m\bar p}g_{m\bar j\bar
l}g_{ik\bar p},\\
\label{eqcur2}
T_{\beta_1\dots\beta_q;\gamma}^{\alpha_1\dots\alpha_p}=\partial_{\gamma}T_{\beta_1\dots\beta_q}^{\alpha_1\dots\alpha_p}-\sum_{i=1}^q
\Gamma_{\gamma\beta_i}^{\delta}T_{\beta_1\dots\beta_{i-1}\delta\beta_{i+1}\dots\beta_q}^{\alpha_1\dots\alpha_p}
+\sum_{j=1}^q\Gamma_{\delta\gamma}^{\alpha_j}T_{\beta_1\dots\beta_q}^{\alpha_1\dots\alpha_{j-1}\delta\alpha_{j+1}\dots\alpha_p}.
\end{gather}
The second equation gives the formula for covariant derivatives of a tensor field
$T_{\beta_1\dots\beta_q}^{\alpha_1\dots\alpha_p}$, where the Christoffel symbols $\Gamma_{\beta\gamma}^\alpha=0$ except
for $
\Gamma_{jk}^i=g^{i\bar l}g_{j\bar l k},\, \Gamma_{\bar j\bar
k}^{\bar i}=g^{l\bar i}g_{l\bar j\bar k}$.

\vskip 30pt
\section{The asymptotic expansion of the Berezin transform} \label{berezin}

We may write \eqref{eqber} as
\begin{equation} \label{eqber2}
\sum^\infty_{k=0} B_k(x) \alpha^{n-k} I_\alpha f(x)=\int_\Omega f(y)
|K_\alpha(x,y)|^2 e^{-\alpha \Phi(x)-\alpha
\Phi(y)}\frac{w_g^n(y)}{n!}.
\end{equation}

By \eqref{eqb1}, we have that for $(x,y)$ near the diagonal,
\begin{equation} \label{eqb2}
|K_\alpha (x,y)|^2 e^{-\alpha\Phi(x)-\alpha \Phi (y)}=e^{-\alpha
D(x,y)} \sum^\infty_{k=0} \sum_{i=0}^k B_{i}(x,y)B_{k-i}(y,x)
\alpha^{n-k}.
\end{equation}

We need the following important result.
\begin{theorem}{\rm(Engli\v{s})} \label{eng}
If the Laplace integral
$$\int_{\Omega} f(y)e^{-mD(x,y)}\frac{\omega_g^n(y)}{n!}$$
exists for some $m=m_0$, then it also exists for all $m>m_0$.
Moreover, it has an asymptotic expansion for $m\rightarrow\infty$,
\begin{equation}
\int_{\Omega} f(y)e^{-mD(x,y)}\frac{\omega_g^n(y)}{n!} \sim
\frac{1}{m^n}\sum_{j\geq0}m^{-j}R_j(f)(x),
\end{equation}
where $R_j: C^\infty(\Omega)\rightarrow C^\infty(\Omega)$ are
explicit differential operators defined by
\begin{equation} \label{eqeng}
R_j f(x)=\frac{1}{\det g}\sum_{k=j}^{3j}\frac{1}{k!(k-j)!}L^k(f\det
g S^{k-j})|_{y=x},
\end{equation}
where $L$ is the (constant-coefficient) differential operator
$$L f(y)=g^{i\bar j}(x)\partial_i\partial_{\bar j} f(y)$$
and the function $S(x,y)$ satisfies $$S=\partial_\alpha
S=\partial_{\alpha\beta}S=\partial_{i_1 i_2\dots
i_m}S=\partial_{\bar i_1\bar i_2\dots \bar i_m}S=0 \quad \text{ at }
y=x,$$
$$\partial_{i\bar j\alpha_1\alpha_2\dots
\alpha_m}S|_{y=x}=-\partial_{\alpha_1\alpha_2\dots \alpha_m}g_{i\bar
j}(x).$$
\end{theorem}

Formulas for $R_0$ and $R_1$ were computed by Berezin \cite{Ber}.
Engli\v{s} \cite{Eng} obtained $R_k,\, k\leq3$ by a tour de force
computation.

\begin{theorem} \label{engr} Let $\Gamma=(V\cup\{f\},E)\in \dot{\mathcal G}^{ss}$. Then we have
\begin{equation} \label{eqr}
R_{\Gamma}=\frac{\det(A(\Gamma_-)-I)}{|\dot{\rm Aut}(\Gamma)|},
\end{equation}
where $\Gamma_-$ is the subgraph of $\Gamma$ obtained by removing
the vertex $f$ from $\Gamma$.
\end{theorem}
\begin{proof} Let $\mathcal L$ be the set of linear directed subgraphs of
$\Gamma_-$. We define an equivalence relation $\sim$ on $\mathcal L$
by
\begin{equation}
L_1\sim L_2 \text{ if there is an automorphism } h\in \dot{\rm
Aut}(\Gamma) \text{ such that } h(L_1)=L_2.
\end{equation}

Let $\tilde{\mathcal L}=\mathcal L/\sim$ be the equivalence class.
Given $L\in \mathcal L$, let $\dot{\rm Aut}(L)$ be the subgroup of
$\dot{\rm Aut}(\Gamma)$ that leaves $L$ invariant. From Lemma
\ref{dettree} and \eqref{eqeng}, we have
\begin{equation}
R_{\Gamma}=\sum_{L\in \tilde{\mathcal L}}\frac{(-1)^{p(L)+|V|}}{|\dot{\rm
Aut}(L)|},
\end{equation}
where $p(L)$ denotes the number of components of $L$.

We have the natural action of $\dot{\rm Aut}(\Gamma)$ on $\mathcal
L$. Then the set of orbits is $\tilde{\mathcal L}$ and the isotropy
group at $L$ is $\dot{\rm Aut}(L)$. So we get the desired equation
\eqref{eqr} from Theorem \ref{cds}.
\end{proof}

From \eqref{eqber2}, \eqref{eqb2} and Theorem \ref{eng}, we get
Loi's recursion formula \cite{Loi}
\begin{equation} \label{eqloi}
B_k(x)=-\sum_{i+j=k \atop
i,j\geq1}B_i(x)B_j(x)-\sum_{\ell+i+j=k\atop 1\leq \ell\leq k}R_\ell
(B_i(x,y)B_j(y,x))|_{y=x}.
\end{equation}
It was pointed out to the author recently that essentially the same identity was also obtained independently
in \cite{Cha}.

In \cite{Xu}, $B_k$ was written as a summation over graphs.
\begin{equation} \label{eqc}
B_k(x)=\sum_{G\in\mathcal G^{ss}(k)} z(G)\cdot G, \quad
z(G)\in\mathbb Q
\end{equation}
and it is proved that if $G=(V,E)\in \mathcal G^{ss}$ is strongly
connected, then
\begin{equation} \label{eqstrong}
 z(G) = -\frac{\det(A-I)}{|{\rm Aut}(G)|},
\end{equation}
where $A$ is the adjacency matrix of $G$. This formula also follows
directly from Theorem \ref{engr} and \eqref{eqloi}.

In general, if $G\in\mathcal G^{ss}$ is a disjoint union of connected
subgraphs $G=G_1\cup\dots\cup G_n$, then we have
\begin{equation} \label{eqdiscon}
z(G)=\begin{cases}\dfrac{(-)^n\det(A-I)}{|{\rm Aut}(G)|}, & \text{if
all } G_i \text{ are strongly connected};
\\
0,& \text{otherwise}.
\end{cases}
\end{equation}
The following proposition was proved in \cite{Xu}. Here we give a
shorter proof using Theorem \ref{engr}.
\begin{proposition} \label{sink}
If $G\in \mathcal G^{ss}$ is connected but not strongly connected, then
$z(G) = 0$.
\end{proposition}
\begin{proof}
We will proceed by induction on the weight of $G$. First we assume that $G\in \mathcal G(k)$ is stable and has weight $k$.
Then any sink or source of $G$ must be at
least semistable.
Without loss of generality, we may assume that $C\in\mathcal G^{ss}$ is a sink of $G$.

In the right-hand side of
\eqref{eqloi}, graphs from the first summation are disconnected, so
they do not contribute to $z(G)$ and may be omitted. For the second summation in the right-hand side of
\eqref{eqloi}, we denote by $\tilde B^+$ and $\tilde B^-$ the source $B_j(y,x)$ and the sink $B_i(x,y)$ respectively.
By induction, both $\tilde B^+$ and $\tilde B^-$ are disjoint union of strongly connected semistable graphs.
By \eqref{eqr}, we see that $z(G)$ equals the coefficient of $G$ in
\begin{equation}\label{eqb6}
-\sum_{\ell+i+j=k\atop 1\leq\ell\leq k}\sum_{\Gamma\in\dot{\mathcal G}(\ell)\atop \tilde B^+\in\mathcal G^{ss}(j),\tilde B^-\in\mathcal G^{ss}(i)}\frac{\det(A(\Gamma_{-})-I)}{|{\dot{\rm Aut}}(\Gamma)|}z(\tilde B^+)z(\tilde B^-)D_{\Gamma}(\tilde B^+,\tilde B^-).
\end{equation}
Note that in the above summation, we have either $C\subset \Gamma_-$ or $C\subset \tilde B^-$.
As illustrated in Figure
\ref{figberg}, where each arrow in the graph may represent
multiple edges, we have $(\Gamma_-,\tilde B^+,\tilde B^-)=(H+C,B^+,B^-)$ in the former case and
$(\Gamma_-,\tilde B^+,\tilde B^-)=(H,B^+,B^-\coprod C)$ in the latter case.

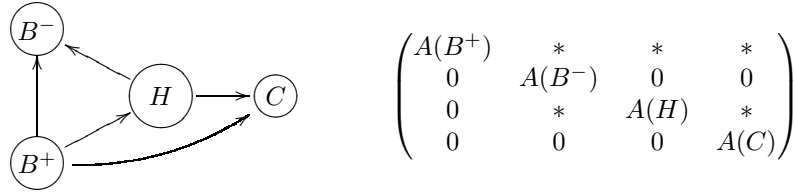
\begin{figure}[h]
\begin{tabular}{c}
$\xymatrix@C=7mm@R=1mm{ *++[o][F-]{B^-} & &\\  &
*+++[o][F-]{H} \ar[ul] \ar[r]& *++[o][F-]{C}
\\ *++[o][F-]{B^+} \ar[ur]\ar[uu]\ar@/_0.8pc/[urr] & &}$
\end{tabular} \qquad
\begin{tabular}{c}
$\begin{pmatrix} A(B^+) & \ast & \ast & \ast
\\ 0& A(B^-) & 0 & 0
\\ 0& \ast & A(H) & \ast
\\ 0 & 0 & 0 & A(C)
\end{pmatrix}$
\end{tabular}
\caption{A configuration for $G$ and the adjacency
matrix} \label{figberg}
\end{figure}

By Lemma \ref{graph2},
the contribution of $(\Gamma_-,\tilde B^+,\tilde B^-)=(H+C,B^+,B^-)$ to $z(G)$ equals
\begin{multline}\label{eqb7}
-z(B^-)|{\rm Aut}(B^-)|\times z(B^+)|{\rm
Aut}(B^+)|\frac{1}{|{\rm
Aut}(G)|}\det\left(\begin{pmatrix} A(H) & \ast
\\ 0& A(C)
\end{pmatrix}-I\right)\\
=-z(B^-)|{\rm Aut}(B^-)|\times z(B^+)|{\rm
Aut}(B^+)|\frac{1}{|{\rm
Aut}(G)|}\det(A(C)-I)\det(A(H)-I)
\end{multline}
and the contribution of $(\Gamma_-,\tilde B^+,\tilde B^-)=(H,B^+,B^-\coprod C)$ to $z(G)$ equals
\begin{multline}\label{eqb8}
-z(B^-\coprod C)\big|{\rm Aut}(B^-\coprod C)\big|\times z(B^+)|{\rm
Aut}(B^+)|\frac{1}{|{\rm
Aut}(G)|}\det(A(H)-I)\\
=-z(B^-)|{\rm Aut}(B^-)|\times z(C)|{\rm Aut}(C)|\times z(B^+)|{\rm
Aut}(B^+)|\frac{1}{|{\rm
Aut}(G)|}\det(A(H)-I)\\
=-z(B^-)|{\rm Aut}(B^-)|\times (-\det(A(C)-I))\times z(B^+)|{\rm
Aut}(B^+)|\frac{1}{|{\rm
Aut}(G)|}\det(A(H)-I).
\end{multline}
In the second equation of \eqref{eqb8}, we used
\begin{gather*}
z(B^-\coprod C)=\frac{z(B^-)z(C)}{\varepsilon(B^-,C)},\\
\big|{\rm Aut}(B^-\coprod C)\big|=|{\rm
Aut}(B^-)||{\rm
Aut}(C)|\varepsilon(B^-,C),
\end{gather*}
where $\varepsilon(B^-,C)=1$ or $2$ depending on whether $B^-\ncong C$ or $B^-\cong C$.

Therefore for any given $B^+$ and $B^-$, we have that \eqref{eqb7} and \eqref{eqb8} add up to zero. This concludes the proof
of $z(G)=0$ when $G$ is a connected but not strongly connected stable graph.

When $G$ is only semistable, $z(G)=0$ follows from the fact that when using \eqref{eqcur1} and \eqref{eqcur2} to convert
between covariant
derivatives of curvatures and partial derivatives
of metrics, we always turn strongly connected graphs into strongly connected graphs.
\end{proof}

Fix a bounded neighborhood $U$ of $x$ such that \eqref{eqb2} holds.
Applying Theorem \ref{eng} to the right-hand side of \eqref{eqber2}
and using \eqref{eqber3}, we get
\begin{equation} \label{eqb3}
\sum_{m=0}^k B_{m}(x) Q_{k-m} f(x)=\sum^k_{j=0}\sum_{i=0}^{k-j} R_j
(B_{i}(x,y)B_{k-j-i}(y,x)f(y))|_{y=x},
\end{equation}
namely
\begin{equation} \label{eqb5}
Q_k f(x)=\sum^k_{j=0}\sum_{i=0}^{k-j} R_j
(B_{i}(x,y)B_{k-j-i}(y,x)f(y))|_{y=x}-\sum_{m=1}^k B_{m}(x) Q_{k-m}
f(x),
\end{equation}
where the operators $R_j$ apply to the $y$-variable.

\begin{theorem} \label{engq}
Let $\Gamma\in \dot{\mathcal G}$. Then
\begin{equation} \label{eqq}
Q_{\Gamma}=
\begin{cases}
\dfrac{\det(A(\Gamma_-)-I)}{|\dot{\rm Aut}(\Gamma)|} & \text{if }
\Gamma \text{ is strongly connected},\\
0 & otherwise,
\end{cases}
\end{equation}
where $\Gamma_-$ is the subgraph of $\Gamma$ obtained by removing
the vertex $f$ from $\Gamma$.
\end{theorem}
\begin{proof}
We will use induction on the weight of $\Gamma$. There are three
cases:
\begin{enumerate}
\item[i)]
Assume that $\Gamma\in \dot{\mathcal G}$ is a disjoint union of
connected subgraphs $\Gamma=\Gamma_1\cup\dots\cup \Gamma_n,\, n\geq
1$ and $\Gamma_1$ is not strongly connected. Since $\Gamma$ is
stable, $\Gamma_1$ must have a source or sink that does not contain
the distinguished vertex $f$ and belongs to $\mathcal G^{ss}$.

Using the same argument in the proof of Proposition \ref{sink}, we
can prove that the contribution of the first term in the right-hand
side of \eqref{eqb5} to $Q_\Gamma$ is zero. The contribution of the
second term in the right-hand side of \eqref{eqb5} to $Q_\Gamma$ is
also zero by induction and Proposition \ref{sink}.

\item[ii)]
Assume that each component in $\Gamma=\Gamma_1\cup\dots\cup
\Gamma_n,\, n\geq 2$ is strongly connected and $f\in \Gamma_n$. Then
by Theorem \ref{engr} and \eqref{eqdiscon}, we see that the
contribution of the first term in the right-hand side of
\eqref{eqb5} to $Q_\Gamma$ is
\begin{equation} \label{eqb4}
 \frac{(-1-1+1)^{n-1}}{|\dot{\rm
Aut}(\Gamma)|}\det(A(\Gamma_-)-I).
\end{equation}
By induction, we see that the contribution of the second term in the
right-hand side of \eqref{eqb5} to $Q_\Gamma$ is
\begin{equation}
 \frac{(-1)^{n-1}}{|\dot{\rm
Aut}(\Gamma)|}\det(A(\Gamma_-)-I),
\end{equation}
which cancel with \eqref{eqb4}. So combining (i), we have proved
that $Q_{\Gamma}=0$ if $\Gamma$ is not connected.

\item[iii)]
If $\Gamma$ is strongly connected, then the contribution of the
second term in the right-hand side of \eqref{eqb5} to $Q_\Gamma$ is
$0$. the contribution of the first term in the right-hand side of
\eqref{eqb5} to $Q_\Gamma$ is just
$$R_{\Gamma}=\frac{\det(A(\Gamma_-)-I)}{|\dot{\rm Aut}(\Gamma)|}.$$
\end{enumerate}

So we conclude the proof with the above three cases.
\end{proof}

\begin{corollary}
The equation \eqref{eqq} holds also for $\Gamma\in \dot{\mathcal G}^{ss}$.
\end{corollary}
\begin{proof}
When $\Gamma\in\dot{\mathcal G}^{ss}$ is strongly connected, $Q_\Gamma=R_\Gamma$ still holds. When $\Gamma\in\dot{\mathcal G}^{ss}$
is connected but not strongly connected, $Q_\Gamma=0$ follows by the same reason as stated in the end of proof of Proposition \ref{sink}.
\end{proof}

\begin{corollary}
Given $k\geq 0$. Let $\Gamma=\Big[\xymatrix{\bullet
\ar@(ur,dr)[]^{k}}\Big]$. Then
\begin{equation}
R_{\Gamma}=Q_{\Gamma}=\frac{1}{k!}.
\end{equation}
\end{corollary}

Engli\v{s} \cite{Eng} defined a scalar invariant $r_k=R_k(1)$. We
see that only graphs with $\deg^+f=\deg^-f=0$ contribute to $r_k$.
\begin{corollary} We have $r_0=1$ and when $k\geq 1$  (see Remark \ref{rm1})
\begin{equation}
r_k=\sum_{G\in\mathcal G(k)} r(G) G=\sum_{G\in\mathcal G(k)}
\frac{\det(A(G)-I)}{|{\rm Aut}(G)|} G.
\end{equation}
\end{corollary}

Thus we have $r(G)=(-1)^{n(G)}z(G)$ when each of the $n(G)$
components of $G$ is strongly connected.

The Theorem \ref{engq} also proves Theorem \ref{main}. It is obvious
that the Berezin star product defined in \eqref{eqmain} is local in
the sense that $\supp C_j(f_1, f_2)$ is contained in $\supp f_1
\cap\supp f_2$ for all $j\geq1$. Since $Q_\Gamma=0$ if $\Gamma$ is
not strongly connected, we see that \eqref{eqmain} defines a
deformation quantization with separation of variables, namely it
satisfies $f\star h=f\cdot h$ and $h\star g=h\cdot g$ for any
locally defined holomorphic function $f$, antiholomorpihc function
$g$ and an arbitrary function $h$. In particular, $1$ is the unit in
the star product.

\begin{proposition}
The Berezin star product satisfies $\overline{f_1\star
f_2}=\bar{f_2}\star\bar{f_1}$.
\end{proposition}
\begin{proof}
We have
\begin{align*}
\overline{f_1\star
f_2}(x)&=\sum_{\Gamma=(V\cup\{f\},E)\in\dot{\mathcal
G}_{scon}}\frac{\det(A(\Gamma_-)-I)}{\dot{\rm Aut}(\Gamma)}
h^{|E|-|V|} D_{\Gamma^T}(\bar f_2,\bar f_1)\Big|_x\\
&=\sum_{\Gamma=(V\cup\{f\},E)\in\dot{\mathcal
G}_{scon}}\frac{\det(A(\Gamma^T_-)-I)}{\dot{\rm Aut}(\Gamma^T)}
h^{|E|-|V|} D_{\Gamma^T}(\bar f_2,\bar f_1)\Big|_x\\
&=\bar{f_2}\star\bar{f_1}.
\end{align*}
Here $\Gamma^T$ is the transpose of $\Gamma$, namely $\Gamma^T$ is
obtain by reversing all arrows in $\Gamma$.
\end{proof}

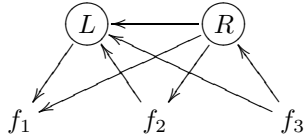
\begin{figure}[h]
$\xymatrix@C=3mm@R=7mm
{& *++[o][F-]{L}\ar[dl] & &*++[o][F-]{R}\ar[ll] \ar[dl] \ar[dlll] &\\
 f_1 & &f_2 \ar[ul] & &f_3 \ar[ul]\ar[ulll]} $ \caption{A graph $\Gamma$ with 3 distinguished vertices} \label{figassoc}
\end{figure}

\begin{proposition} The Berezin star product is associative, namely
$f_1\star(f_2\star f_3)=(f_1\star f_2)\star f_3$.
\end{proposition}
\begin{proof}
We have
\begin{equation}
f_1\star f_2=\sum_{\Gamma=(V\cup\{f\},E)\in\dot{\mathcal G}^{ss}}
Q_\Gamma h^{|E|-|V|} D_{\Gamma}(f_1, f_2).
\end{equation}
As we are taking summation over pointed semistable graphs, this
equation holds in a neighborhood of $x$.

Since the associativity is equivalent to
\begin{equation}
\sum_{j=0}^k C_j(f_1,C_{k-j}(f_2,f_3))=\sum_{j=0}^k
C_{k-j}(C_j(f_1,f_2),f_3),
\end{equation}
it is enough to prove that for any two pointed semistable graphs
$L,R$ in $\dot{\mathcal G}^{ss}$ as shown in
Figure \ref{figassoc}, the coefficients of any given stable graph $\Gamma$ in $
Q_{L}D_{L}(f_1,Q_{R}D_R(f_2,f_3))$ and $Q_{R}D_{R}(Q_L
D_L(f_1,f_2),f_3)$ are equal. In fact, by Lemma \ref{graph2} and
Remark \ref{rm2}, the coefficients of a stable graph $\Gamma$ in both terms are equal to
\begin{equation}
Q_{L}Q_R|\dot{\rm Aut}(L)||\dot{\rm Aut}(R)|\frac{1}{|{\rm
Aut}(\Gamma)|}.
\end{equation}
Note that $\Gamma$ has three distinguished vertices labeled by $f_1,f_2$
and $f_3$. An automorphism in ${\rm Aut}(\Gamma)$ should fix
these three vertices.
\end{proof}

\vskip 30pt
\section{Computations of $R_k, Q_k,\, k\leq3$}

Fix a normal coordinate around $x\in M$. The covariant derivative of
$f$ satisfies
\begin{equation}
f_{;\alpha_1\dots\alpha_p\beta}=\partial_\beta
f_{;\alpha_1\dots\alpha_p}-\sum_{i=1}^p
\Gamma_{\beta\alpha_i}^\gamma
f_{;\alpha_1\dots\alpha_{i-1}\gamma\alpha_{i+1}\dots\alpha_p}.
\end{equation}
We can use the above equation to write partial derivatives of $f$ in
terms of covariant derivatives of $f$.

\begin{example}\label{computeq} The formulae of $R_k$ and $Q_k$ in partial derivatives
may be computed readily using \eqref{eqengrq}, Theorem \ref{engr}
and Theorem \ref{engq}. We will convert them to the curvature-tensor
expressions. Note that our convention of curvatures $R_{i\bar j
k\bar l}, R_{i\bar j}, \rho$ in \cite{Xu} all differ by a minus sign
with that of \cite{Eng}. The following notations were introduced in
\cite{Eng}.
\begin{equation}
L_{Ric}f:=R_{i \bar j}f_{;j\bar i},\qquad L_{R} f:=R_{i\bar j k\bar
l}f_{;j\bar i l\bar k}.
\end{equation}

When $k=0$, $R_0=Q_0=1$.

When $k=1$, we have
\begin{align*}
R_1 f= \Big[\xymatrix{\bullet \ar@(ur,dr)[]^{1}}\Big]
+\frac{1}{2}\left[\xymatrix{*+[o][F-]{2}}\right]&=f_{i\bar i}+\frac12 g_{i\bar i j\bar j}f\\
&=f_{;i\bar i}-\frac12 \rho f\\
&=\Delta f-\frac12 \rho f.
\end{align*}
\begin{align*}
Q_1 f=\Big[\xymatrix{\bullet \ar@(ur,dr)[]^{1}}\Big]&=f_{i\bar i}\\
&=\Delta f.
\end{align*}

When $k=2$, we have
\begin{align*}
R_2 f=& \frac12\Big[\xymatrix{\bullet
\ar@(ur,dr)[]^{2}}\Big]+\frac12 \Big[\xymatrix{*+[o][F-]{2}}\,\Big{|}\xymatrix{\bullet \ar@(ur,dr)[]^{1}}\Big]+\frac12
\Big[\xymatrix{*+[o][F-]{2} & \bullet \ar[l]}\Big]+\frac12
\Big[\xymatrix{*+[o][F-]{2} \ar[r] &
\bullet }\Big]\\
& \qquad+\frac13 \left[\xymatrix{*+[o][F-]{3}}\right]-\frac38
\left[\xymatrix{ \circ \ar@/^/[r]^2 & \circ \ar@/^/[l]^2
         }\right] -\frac 12 \left[\xymatrix{
        *+[o][F-]{1} \ar@/^/[r]^1
         &
        *+[o][F-]{1} \ar@/^/[l]^1} \right ] +\frac{1}{8}\left[\xymatrix{*+[o][F-]{2}}\mid \xymatrix{*+[o][F-]{2}}\right]\\
=& \frac12 f_{i\bar i j\bar j}+\frac12 g_{i\bar i j\bar j}f_{k\bar
k}+\frac12 g_{i\bar i j\bar j\bar k}f_k+\frac12 g_{i\bar i j\bar j
k}f_{\bar k}&\\
&\qquad +\left(\frac13 g_{i\bar i j\bar j k\bar k}-\frac38 g_{i\bar
j k\bar l}g_{j\bar i l\bar k}-\frac12 g_{i\bar i k\bar l}g_{j\bar j
l\bar k}+\frac18 g_{i\bar i j\bar j}g_{k\bar k l\bar l}\right)f\\
=& \frac12\Delta^2 f-\frac12 L_{Ric}f-\frac{\rho}{2}\Delta f-\frac12
(\rho_{;\bar k}f_{;k}+\rho_{;k}f_{;\bar
k})\\
&\qquad
+\left(-\frac13\Delta\rho-\frac1{24}|R|^2+\frac16|Ric|^2+\frac18\rho^2\right)f.
\end{align*}
and
\begin{align*}
Q_2 f=\frac12\Big[\xymatrix{\bullet \ar@(ur,dr)[]^{2}}\Big]&=\frac12f_{i\bar i j\bar j}\\
&=\frac12 f_{;i\bar i j\bar j}-\frac12 R_{i\bar k}f_{;k\bar i}\\
&=\frac12\Delta^2 f-\frac12 L_{Ric}f.
\end{align*}

When $k=3$, we express $Q_3 f$ in terms of the basis as used in
Engli\v{s} \cite{Eng}.
\begin{gather}
\sigma_1=\Delta^3 f,\quad \sigma_2=R_{i\bar j}(\Delta f)_{;j\bar
i},\quad \sigma_3=R_{i\bar j k\bar l}f_{;j\bar i l\bar k},\nonumber\\
\sigma_4=R_{i\bar j;\bar k}f_{;j\bar i k},\quad \sigma_5=R_{i\bar j;
k}f_{;j\bar i \bar k},\quad \sigma_6=R_{i\bar j k\bar l}R_{j\bar i m\bar k}f_{;l\bar m},\label{gathertensor}\\
\sigma_7=R_{i\bar j k\bar l}R_{j\bar i}f_{;l\bar k},\quad
\sigma_8=\rho_{;i\bar j}f_{;j\bar i},\quad \sigma_9=R_{i\bar
j}R_{k\bar i}f_{;j\bar k}.\nonumber
\end{gather}

We will compute the coefficients $c_i,\,1\leq i \leq 9$, such that
\begin{equation} \label{eqthreeinr}
Q_3 f=c_1\sigma_1+c_2\sigma_2+\cdots+c_{9}\sigma_{9}.
\end{equation}

There are $9$ strongly connected pointed stable graphs of weight $3$
in $\dot{\mathcal G}_{scon}(3)$.
\begin{gather}
\tau_1=\Big[\xymatrix{\bullet \ar@(ur,dr)[]^{3}}\Big],\quad
\tau_2=\left[\xymatrix{
        *+[o][F-]{1} \ar@/^/[r]^1
         &
        \bullet\, 1 \ar@/^/[l]^1} \right ],\quad
\tau_3=\left[\xymatrix{ \circ \ar@/^/[r]^2 & \bullet \ar@/^/[l]^2
         }\right],\nonumber\\
\tau_4=\left[\xymatrix{
        *+[o][F-]{1} \ar@/^/[r]^1
         &
        \bullet \ar@/^/[l]^2} \right ],\quad
\tau_5=\left[\xymatrix{
        *+[o][F-]{1} \ar@/^/[r]^2
         &
        \bullet \ar@/^/[l]^1} \right ],\quad
\tau_6=\left[\begin{minipage}{0.6in}$\xymatrix@C=2mm@R=5mm{
                & \bullet \ar[dr]^{1}             \\
         \circ \ar[ur]^{1} \ar@/^0.3pc/[rr]^{1} & &  \circ  \ar@/^0.3pc/[ll]^{2}
         }$
         \end{minipage}\right],\label{gathergraph}\\
\tau_7=\left[\begin{minipage}{0.65in}
             $\xymatrix@C=2mm@R=5mm{
                & \bullet \ar@/^-0.3pc/@<0.2ex>[dl]^{1}             \\
         \circ \ar@/^0.3pc/@<0.7ex>[ur]^{1} \ar@/^-0.3pc/@<0.2ex>[rr]^{1} & & *+[o][F-]{1}
         \ar@/^0.3pc/@<0.7ex>[ll]^{1}
         }$
         \end{minipage}\right],\quad
\tau_8=\left[\xymatrix{
        *+[o][F-]{2} \ar@/^/[r]^1
         &
        \bullet \ar@/^/[l]^1} \right ],\quad
\tau_9=\left[\begin{minipage}{0.6in}$\xymatrix@C=2mm@R=5mm{
                & \bullet \ar[dr]^{1}             \\
        *+[o][F-]{1} \ar[ur]^{1}  & &  *+[o][F-]{1} \ar[ll]^{1}
         }$
         \end{minipage}\right]. \nonumber
\end{gather}

By Theorem \ref{engq}, we have
\begin{equation} \label{eqthreeing}
Q_3 f= q_1\tau_1+q_2\tau_2+\cdots+q_{9}\tau_{9},
\end{equation}
where
\begin{gather*}
q_1=1/6,\quad q_2=0,\quad
q_3=-1/4,\quad q_4=0,\quad q_5=0,\\
q_6=-1/2,\quad q_7=-1,\quad q_8=1/2,\quad q_9=0.
\end{gather*}

We need to express each $\tau_i$ as a linear combination of
$\sigma_i,\,1\leq i\leq 9$. By a tedious but straightforward
computation, we get
\begin{gather*}
\tau_1
=\sigma_1+3\sigma_2+2\sigma_3+2\sigma_4+2\sigma_5+\sigma_6+4\sigma_7+\sigma_8-2
\sigma_9, \\
\tau_2 = \sigma_2+\sigma_7,\quad \tau_3 = \sigma_3+\sigma_6,\quad
\tau_4 =
\sigma_4,\quad \tau_5 = \sigma_5, \\
\tau_6 = \sigma_6,\quad \tau_7 = \sigma_7,\quad \tau_8 =
\sigma_6+2\sigma_7+\sigma_8,\quad \tau_9 = \sigma_9.
\end{gather*}
Substituting into \eqref{eqthreeing}, we can get the coefficients in
\eqref{eqthreeinr}.
\begin{gather*}
c_1=1/6,\quad c_2=-1/2,\quad
c_3=-1/12,\quad c_4=-1/3,\quad c_5=-1/3,\\
c_6=-1/12,\quad c_7=2/3,\quad c_8=-2/3,\quad c_9=-1/3.
\end{gather*}

All these values of $R_k, Q_k$ computed here match the computations
by Engli\v{s} \cite{Eng}.

From $Q_k,\,k\leq 3$, we can get the invariant expressions for the
first four coefficients of the Berezin star product
\eqref{eqberdef}.
\begin{align*}
C_0(f_1,f_2)&=f_1 f_2, \\
C_1(f_1,f_2)&={f_1}_{;\bar i} {f_2}_{;i},\\
C_2(f_1,f_2)&=\frac12 {f_1}_{;\bar i\bar j}{f_2}_{;i j},\\
C_3(f_1,f_2)&=\frac16 {f_1}_{;\bar i\bar j\bar
k}{f_2}_{;ijk}+\frac14 R_{i\bar j
k\bar l}{f_1}_{;\bar i\bar k}{f_2}_{;jl}-\frac12 R_{i\bar j k\bar l}R_{j\bar i m\bar k}{f_1}_{;\bar m}{f_2}_{;l}\\
& \qquad -R_{i\bar j k\bar l}R_{j\bar i}{f_1}_{;\bar k}{f_2}_{;
l}-\frac12 \rho_{;i\bar j}{f_1}_{;\bar i}{f_2}_{;j}.
\end{align*}
Note that around a normal coordinate system of $x$, we have
$f_{i_1,\dots i_r}(x)=f_{;i_1,\dots i_r}(x)$ and $f_{\bar i_1,\dots
\bar i_r}(x)=f_{;\bar i_1,\dots \bar i_r}(x)$ for any $r\geq1$.

\end{example}
\begin{example}\label{computeq2}
We now describe how to compute $R_3$ in terms of curvature tensors
explicitly. The method works for any $R_k$ or $Q_k$ and for the
Bergman kernel.

By Table \ref{tb1}, there are $46$ pointed stable graphs of weight
$3$ in $\dot{\mathcal G}(3)$. Let $\tau_i,\,1\leq i\leq 46$ be the
corresponding Weyl invariants in terms of partial derivatives of
metrics. They are also in one-to-one correspondence with a basis
$\sigma_i,\,1\leq i\leq 46$ of curvature tensors of weight $3$. Each
representative $\sigma_i$ is determined up to an interchange of
indices. By Ricci formula, the difference lies in the space of
strictly lower degree curvature tensors. (The definition of weight
and order of a curvature tensor can be found in \cite{Lu} or
\cite{Xu}).

It is relatively easy to express $\sigma_i$ in terms of $\tau_i$.
Namely we can obtain a $46\times 46$ square matrix
$M=[m_{ij}]_{1\leq i, j\leq 46}$ of rational numbers, such that
\begin{equation}
\sigma_i=\sum_{j=1}^{46} m_{ij}\tau_j,\quad 1\leq i\leq 46.
\end{equation}
Let $\tilde M=[\tilde m_{ij}]_{1\leq i, j\leq 46}$ be the inverse
matrix of $M$, then
\begin{equation} \label{eqtautosigma}
\tau_i=\sum_{j=1}^{46} \tilde m_{ij}\sigma_j,\quad 1\leq i\leq 46.
\end{equation}

By Theorem \ref{engr} and \eqref{eqtautosigma}, we finally get the
curvature-tensor expression for $R_3 f$.
\begin{align*}
R_3 f&=\sum_{i=1}^{46}\frac{\det(A((\tau_i)_-)-I)}{|\dot{\rm
Aut}(\tau_i)|}\tau_i\\
&=\sum_{i=1}^{46}\sum_{j=1}^{46} \tilde
m_{ij}\frac{\det(A((\tau_i)_-)-I)}{|\dot{\rm Aut}(\tau_i)|}\sigma_i.
\end{align*}

We implemented the above procedure with the help of a computer and
the final result of $R_3 f$ matches with that computed by Engli\v{s}
\cite{Eng}.
\end{example}

From Example \ref{computeq} and the appendix, we have computed
$Q_k,\,0\leq k\leq 4$, thus the first five terms of the Berezin star
product \eqref{eqmain}.

\begin{example} \label{computebt}
We now compute $C^{BT}_j,\,j\leq 3$. By \eqref{eqbt}, we see that
the coefficients of $C^{BT}_j(f_1,f_2)$ equal to the coefficients in
the asymptotic expansion of $I^{-1}$. The latter can be computed
using \eqref{eqmain}.
\begin{multline}
I^{-1}=f-h\Big[\xymatrix{\bullet
              \ar@(ur,dr)[]^{1}}\Big]
+h^2\left(\frac12\Big[\xymatrix{\bullet
           \ar@(ur,dr)[]^{2}}\Big]-\left[\xymatrix{
        *+[o][F-]{1} \ar@/^/[r]^1
         &
        \bullet \ar@/^/[l]^1} \right ]\right)\\
+h^3\left(-\frac{1}{6}\tau_1+\tau_2+\frac{1}{4}\tau_3+\frac{1}{2}\tau_4+\frac{1}{2}\tau_5-\tau_9\right)+O(h^4).
\end{multline}
These $\tau_i$ are graphs defined in \eqref{gathergraph}.

Converting partial derivatives to covariant derivatives, we get the
invariant expressions for the first four coefficients of the
Berezin-Toeplitz star product \eqref{eqbt}.
\begin{align*}
C^{BT}_0(f_1,f_2)&=f_1 f_2, \\
C^{BT}_1(f_1,f_2)&=-{f_1}_{;i} {f_2}_{;\bar i},\\
C^{BT}_2(f_1,f_2)&=\frac12 {f_1}_{;ij}{f_2}_{;\bar i\bar j}+R_{i\bar j}{f_1}_{;j}{f_2}_{;\bar i},\\
C^{BT}_3(f_1,f_2)&=-\frac16 {f_1}_{;ijk}{f_2}_{;\bar i\bar j\bar
k}-R_{i\bar j}{f_1}_{;jk}{f_2}_{;\bar i\bar k}-\frac14 R_{i\bar j
k\bar l}{f_1}_{;jl}{f_2}_{;\bar i\bar
k}\\
& \qquad -\frac12 R_{i\bar j;\bar k}{f_1}_{;jk}{f_2}_{;\bar
i}-\frac12 R_{i\bar j; k}{f_1}_{;j}{f_2}_{;\bar i\bar k}-R_{i\bar
j}R_{k\bar i}{f_1}_{;j}{f_2}_{;\bar k}.
\end{align*}
Compared to the Berezin star product, the holomorphic and
antiholomorphic variables are swapped.
\end{example}

\vskip 30pt
\section{Fefferman's invariants}

Fefferman \cite{Fef} and Boutet de Monvel-Sj\"ostrand \cite{BS}
proved the asymptotic expansion of the Bergman kernel of a strongly
pseudoconvex domain in $\mathbb C^n$ near the boundary. Nakazawa
\cite{Nak} obtained an explicit formula for the first several
coefficients in Fefferman's asymptotic expansion for bounded
strictly pseudoconvex complete Reinhardt domains in $\mathbb C^2$.

In \cite{Eng, Eng2}, Engli\v{s} established a relation of
Fefferman's invariants with the scalar invariants from the
asymptotic expansion of the Laplace integral, and generalized
Nakazawa's result to arbitrary Hartogs domains in $\mathbb C^n$.

In this section, we will apply Engli\v{s}' work to express
Fefferman's invariants as graph invariants. First we state the
following analogue of Theorem \ref{eng}, again due to Engli\v{s}
\cite{Eng}.
\begin{theorem}{\rm(Engli\v{s})} \label{eng2} Let $\Omega$ be a strongly pseudoconvex domain in $\mathbb C^n$
with real analytic boundary. Then there is an asymptotic expansion for the Laplace integral
$$\int_{\Omega} f(y)e^{-mD(x,y)}\frac{|\det g(x,y)|^2}{\det g(y)}dy \sim \frac{1}{m^n}\sum_{j\geq0}m^{-j}R'_j(f)(x),$$
where $\det g(x,y)$ is the almost analytic extension of $\det g(x)$,
that is
\begin{equation}
\det g(x,y)=\det\left(\frac{1}{\pi}\frac{\partial^2
\Phi(x,y)}{\partial x_j\partial\bar y_k}\right),
\end{equation}
and $R'_j: C^\infty(\Omega)\rightarrow C^\infty(\Omega)$ are
explicit differential operators defined by
\begin{equation} \label{eqfef10}
R'_j f(x)=\frac{1}{(\det
g)^2}\sum_{k=j}^{3j}\frac{1}{k!(k-j)!}L^k(f(y)|\det g(x,y)|^2
S(x,y))|_{y=x},
\end{equation}
where $L$ is the (constant-coefficient) differential operator
$$L f(y)=g^{i\bar j}(x)\partial_i\partial_{\bar j} f(y)$$
and the function $S(x,y)$ is given in Theorem \ref{eng}.
\end{theorem}

We denote by $K'_\alpha(x,y)$ the reproducing kernel of the weighted
Bergman space of all holomorphic function on $\Omega$
square-integrable with respect to the measure $e^{-\alpha \Phi}dx$.
Locally, $K'_\alpha(x,y)$ has an asymptotic expansion in a small
neighborhood of the diagonal  (see \cite{Eng})
\begin{equation}\label{eqfef9}
K'_\alpha (x,y)\sim e^{\alpha \Phi (x,y)} \det
g(x,y)\sum^\infty_{k=0} B'_k (x,y) \alpha^{n-k}.
\end{equation}
The corresponding Berezin transform is given by
\begin{equation}
I'_\alpha f(x)=\int_\Omega
f(y)\frac{|K'_\alpha(x,y)|^2}{K'_\alpha(x,x)}e^{-\alpha \Phi(y)}dy,
\end{equation}
which has an asymptotic expansion
\begin{equation}
I'_\alpha f(x)=\sum^\infty_{k=0} Q'_k f(x)\alpha^{-k}.
\end{equation}
The following analogues of \eqref{eqb5} and \eqref{eqloi} still
hold.
\begin{equation} \label{eqfef}
Q'_k f(x)=\sum^k_{j=0}\sum_{i=0}^{k-j} R'_j
(B'_{i}(x,y)B'_{k-j-i}(y,x)f(y))|_{y=x}-\sum_{m=1}^k B'_{m}(x)
Q'_{k-m} f(x),
\end{equation}
\begin{equation} \label{eqfef2}
B'_k(x)=-\sum_{i+j=k \atop
i,j\geq1}B'_i(x)B'_j(x)-\sum_{\ell+i+j=k\atop 1\leq \ell\leq
k}R'_\ell (B'_i(x,y)B'_j(y,x))|_{y=x}.
\end{equation}

Let $\tilde K(z,\zeta)$ be the (ordinary unweighted) Bergman kernel
of the Hartogs domain
\begin{equation}
\tilde \Omega=\{z=(z_1,z_2)\in\Omega\times \mathbb C^d:\,||
z_2||^2<e^{-\Phi(z_1)}\}.
\end{equation}
It was shown in \cite{Eng2} and \cite{Lig} that
\begin{equation}
\tilde K((z_1,z_2),(\zeta_1,\zeta_2))
=\sum_{j=0}^{\infty}\frac{(j+d)!}{j!\pi^d}
K'_{j+d}(z_1,\zeta_1)\langle z_2,\zeta_2\rangle^k,
\end{equation}
with the convergence uniform on compact subsets. Setting $\zeta=z$,
we have
\begin{equation}
\tilde K(z_1,z_2)=\sum_{j=0}^{\infty}\frac{(j+d)!}{j!\pi^d}
K'_{j+d}(z_1,z_1)||z_2||^{2j}.
\end{equation}
By \eqref{eqfef9}, the coefficients of this asymptotic expansion
(Fefferman's invariants) are determined by $B'_k$. See \cite{Eng}
for a precise description of its behavior when $z$ approaches the
boundary of $\tilde\Omega$.

As pointed out by Engli\v{s}, in a normal coordinate around $x$,
$R'_j$ in \eqref{eqfef10} simplifies to
\begin{equation} \label{eqfef3}
R'_j f(x)=\sum_{k=j}^{2j}\frac{1}{k!(k-j)!}L^k(f S)|_{y=x},
\end{equation}
which means that if we restrict to stable graphs, we do not need to
consider their linear subgraphs when computing $R'_{\Gamma},
Q'_{\Gamma}$ below
\begin{equation} \label{eqfef4}
R'_k f=\sum_{\Gamma\in  \dot{\mathcal G}(k)}R'_{\Gamma}\Gamma,
\qquad Q'_k f=\sum_{\Gamma\in  \dot{\mathcal
G}(k)}Q'_{\Gamma}\Gamma.
\end{equation}

\begin{theorem} \label{engfef} Let $\Gamma=(V\cup\{f\},E)\in \dot{\mathcal G}$. Then we have
\begin{equation}\label{eqfef7}
R'_{\Gamma}=\frac{(-1)^{|V|}}{|\dot{\rm Aut}(\Gamma)|}
\end{equation}
and
\begin{equation}
 Q'_{\Gamma}=
\begin{cases}
\dfrac{(-1)^{|V|}}{|\dot{\rm Aut}(\Gamma)|} & \text{if }
\Gamma \text{ is strongly connected},\\
0 & otherwise.
\end{cases}
\end{equation}
\end{theorem}
\begin{proof}
The formula for $R'_{\Gamma}$ is obvious. The formula for
$Q'_{\Gamma}$ follows from \eqref{eqfef}, \eqref{eqfef7} and
\eqref{eqfef6} by using the same argument in the proof of Theorem
\ref{engq}.
\end{proof}

\begin{corollary} Let $r'_k=R'_k(1)$. Then we have $r'_0=1$ and when $k\geq 1$  (see Remark \ref{rm1})
\begin{equation}
r'_k=\sum_{G\in\mathcal G(k)} r'(G) G=\sum_{G\in\mathcal G(k)}
\frac{(-1)^{|V|}}{|{\rm Aut}(G)|} G.
\end{equation}
\end{corollary}

We may write $B'_k$ as a summation over stable graphs.
\begin{equation} \label{eqfef5}
B'_k(x)=\sum_{G\in\mathcal G(k)} z'(G)\cdot G, \quad z'(G)\in\mathbb
Q.
\end{equation}
\begin{corollary}
If $G=(V,E)\in\mathcal G$ is a disjoint union of connected subgraphs
$G=G_1\cup\dots\cup G_n$, then we have
\begin{equation} \label{eqfef6}
z'(G)=\begin{cases}\dfrac{(-1)^{|V|+n}}{|{\rm Aut}(G)|}, & \text{if
all } G_i \text{ are strongly connected},
\\
0,& \text{otherwise}.
\end{cases}
\end{equation}

\end{corollary}
\begin{proof}
When $G$ is strongly connected, we can use \eqref{eqfef2} and
\eqref{eqfef7} to prove that
\begin{equation} \label{eqfef8}
z'(G)=\frac{(-1)^{|V|+1}}{|{\rm Aut}(G)|}.
\end{equation}

In general, if $G=G_1\cup\dots\cup G_n$ is disjoint union of
connected subgraphs and some $G_i$ is not strongly connected, we can
use the same argument of Proposition \ref{sink} to prove that
$z'(G)=0$. If all $G_i$ are strongly connected, it follows from
Lemma \ref{graph2} and \eqref{eqfef7} that
\begin{equation}
z'(G)=\prod_{j=1}^n z'(G_j)/|Sym(G_1,\dots,G_n)|,
\end{equation}
where $Sym(G_1,\dots,G_m)$ denote the permutation group of these $n$
connected subgraphs. So we conclude the proof of the formula
\eqref{eqfef6}.
\end{proof}

\vskip 30pt

\appendix

\section{The value of $Q_4$} \label{apfour}

There are $36$ strongly connected pointed stable graphs $\Gamma$ in
$\dot{\mathcal G}_{scon}(4)$ such that $Q_\Gamma\neq 0$, which are
listed in Table \ref{tb2}. There are $25$ strongly connected pointed
stable graphs $\Gamma$ in $\dot{\mathcal G}_{scon}(4)$ such that
$Q_\Gamma=0$, which are listed in Table \ref{tb3}. Thus we can use
the method of Example \ref{computeq2} to get a curvature-tensor
expression for $Q_4$.

\begin{table}[h] \footnotesize
\centering \caption{$Q_\Gamma$ of $\Gamma\in\dot\Lambda(4)$}
\label{tb2}
\begin{tabular}{|c|c|c|c|c|c|}

\hline $\xymatrix{\bullet \,4}$
     & $\xymatrix{
        \circ \ar@/^/[r]^2
         &
        \bullet\,1 \ar@/^/[l]^2} $
     & $\xymatrix{
        \circ \ar@/^/[r]^2
         &
        \bullet \ar@/^/[l]^3} $
     & $\xymatrix{
        \circ \ar@/^/[r]^3
         &
        \bullet \ar@/^/[l]^2} $
     & $\xymatrix{
         *+[o][F-]{2} \ar@/^/[r]^1
         &
        \bullet\, 1 \ar@/^/[l]^1} $
     & $\xymatrix{
         *+[o][F-]{2} \ar@/^/[r]^1
         &
        \bullet \ar@/^/[l]^2} $
\\
\hline $1/24$ & $-1/4$ & $-1/12$ & $-1/12$ & $1/2$ & $1/4$
\\
\hline $\xymatrix{
         *+[o][F-]{2} \ar@/^/[r]^2
         &
        \bullet \ar@/^/[l]^1} $
     & $\xymatrix{
         *+[o][F-]{3} \ar@/^/[r]^1
         &
        \bullet \ar@/^/[l]^1} $
     & \begin{minipage}{0.6in}$\xymatrix@C=3mm@R=6mm{
                & \bullet \ar[dr]^{2}             \\
        \circ  \ar[ur]^{2}  & &  \circ \ar[ll]^{2}
         }$
         \end{minipage}
     & \begin{minipage}{0.6in}
             $\xymatrix@C=2mm@R=6mm{
                & \bullet\,1 \ar@/^-0.3pc/@<0.2ex>[dl]^{1}             \\
         \circ \ar@/^0.3pc/@<0.7ex>[ur]^{1} \ar@/^-0.3pc/@<0.2ex>[rr]^{1} & & *+[o][F-]{1}
         \ar@/^0.3pc/@<0.7ex>[ll]^{1}
         }$
         \end{minipage}
     &  \begin{minipage}{0.6in}
             $\xymatrix@C=3mm@R=6mm{
                & \bullet \ar@/^-0.3pc/@<0.2ex>[dl]^{1} \ar[dr]^{1}            \\
         \circ \ar@/^0.3pc/@<0.7ex>[ur]^{1} \ar@/^-0.3pc/@<0.2ex>[rr]^{1} & & *+[o][F-]{1}
         \ar@/^0.3pc/@<0.7ex>[ll]^{1}
         }$
         \end{minipage}
     & \begin{minipage}{0.6in}
             $\xymatrix@C=3mm@R=6mm{
                & \bullet \ar@/^-0.3pc/@<0.2ex>[dl]^{2}             \\
         \circ \ar@/^0.3pc/@<0.7ex>[ur]^{1} \ar@/^-0.3pc/@<0.2ex>[rr]^{1} & & *+[o][F-]{1}
         \ar@/^0.3pc/@<0.7ex>[ll]^{1}
         }$
         \end{minipage}
\\
\hline $1/4$ & $1/3$ & $1/8$ & $-1$ & $-1$ & $-1/2$
\\
\hline  \begin{minipage}{0.6in}
             $\xymatrix@C=3mm@R=6mm{
                & \bullet \ar@/^-0.3pc/@<0.2ex>[dl]^{1}           \\
         \circ \ar@/^0.3pc/@<0.7ex>[ur]^{1} \ar@/^-0.3pc/@<0.2ex>[rr]^{1} & & *+[o][F-]{1}
         \ar@/^0.3pc/@<0.7ex>[ll]^{1} \ar[ul]_{1}
         }$
         \end{minipage}
     & \begin{minipage}{0.6in}
             $\xymatrix@C=3mm@R=6mm{
                & \bullet \ar@/^-0.3pc/@<0.2ex>[dl]^{1}             \\
         \circ \ar@/^0.3pc/@<0.7ex>[ur]^{1} \ar@/^-0.3pc/@<0.2ex>[rr]^{1} & & *+[o][F-]{2}
         \ar@/^0.3pc/@<0.7ex>[ll]^{1}
         }$
         \end{minipage}
  &  \begin{minipage}{0.6in}$\xymatrix@C=2mm@R=6mm{
                & \bullet\,1 \ar[dr]^{1}             \\
         \circ \ar[ur]^{1} \ar@/^0.3pc/[rr]^{1} & &  \circ  \ar@/^0.3pc/[ll]^{2}
         }$
         \end{minipage}
     & \begin{minipage}{0.6in}$\xymatrix@C=3mm@R=6mm{
                & \bullet \ar[dr]^{2}             \\
         \circ \ar[ur]^{1} \ar@/^0.3pc/[rr]^{1} & &  \circ  \ar@/^0.3pc/[ll]^{2}
         }$
         \end{minipage}
     &  \begin{minipage}{0.6in}
             $\xymatrix@C=3mm@R=6mm{
                & \bullet \ar@/^-0.3pc/@<0.2ex>[dl]^{1} \ar[dr]^{1}            \\
         \circ \ar@/^0.3pc/@<0.7ex>[ur]^{1} \ar@/^-0.3pc/@<0.2ex>[rr]^{1} & &
         \circ
         \ar@/^0.3pc/@<0.7ex>[ll]^{2}
         }$
         \end{minipage}
     & \begin{minipage}{0.6in}
             $\xymatrix@C=3mm@R=6mm{
                & \bullet \ar@/^-0.3pc/@<0.2ex>[dl]^{1}             \\
         \circ \ar@/^0.3pc/@<0.7ex>[ur]^{1} \ar@/^-0.3pc/@<0.2ex>[rr]^{2} & &
         \circ \ar[ul]_{1}
         \ar@/^0.3pc/@<0.7ex>[ll]^{1}
         }$
         \end{minipage}
\\
\hline $-1$ & $-1$ & $-1/2$ & $-1/4$ & $-1/2$ & $-1/2$
\\
\hline
\begin{minipage}{0.6in}$\xymatrix@C=3mm@R=6mm{
                & \bullet \ar[dr]^{1}             \\
         \circ \ar[ur]^{1} \ar@/^0.3pc/[rr]^{1} & &  *+[o][F-]{1}  \ar@/^0.3pc/[ll]^{2}
         }$
         \end{minipage}
     & \begin{minipage}{0.6in}
             $\xymatrix@C=3mm@R=6mm{
                & \bullet \ar@/^-0.3pc/@<0.2ex>[dl]^{1}             \\
         \circ \ar@/^0.3pc/@<0.7ex>[ur]^{1} \ar@/^-0.3pc/@<0.2ex>[rr]^{1} & & *+[o][F-]{1}
         \ar@/^0.3pc/@<0.7ex>[ll]^{2}
         }$
         \end{minipage}
  &  \begin{minipage}{0.6in}$\xymatrix@C=3mm@R=6mm{
                & \bullet \ar[dr]^{1}             \\
         \circ \ar[ur]^{1} \ar@/^0.3pc/[rr]^{1} & &  \circ  \ar@/^0.3pc/[ll]^{3}
         }$
         \end{minipage}
     & \begin{minipage}{0.6in}
             $\xymatrix@C=3mm@R=6mm{
                & \bullet \ar@/^-0.3pc/@<0.2ex>[dl]^{1}             \\
         \circ \ar@/^0.3pc/@<0.7ex>[ur]^{2} \ar@/^-0.3pc/@<0.2ex>[rr]^{1} & & *+[o][F-]{1}
         \ar@/^0.3pc/@<0.7ex>[ll]^{1}
         }$
         \end{minipage}
     &  \begin{minipage}{0.6in}$\xymatrix@C=3mm@R=6mm{
                & \bullet \ar[dr]^{1}             \\
         \circ \ar[ur]^{2} \ar@/^0.3pc/[rr]^{2} & &  \circ  \ar@/^0.3pc/[ll]^{1}
         }$
         \end{minipage}
     & \begin{minipage}{0.6in}$\xymatrix@C=3mm@R=6mm{
                & \bullet \ar[dr]^{1}             \\
         *+[o][F-]{1} \ar[ur]^{1} \ar@/^0.3pc/[rr]^{1} & &  \circ  \ar@/^0.3pc/[ll]^{2}
         }$
         \end{minipage}
\\
\hline $-1$ & $-1$ & $-1/3$ & $-1/2$ & $-1/4$ & $-1$
\\
\hline
\begin{minipage}{0.6in}
             $\xymatrix@C=3mm@R=6mm{
                & \bullet \ar@/^-0.3pc/@<0.2ex>[dl]^{1}             \\
         \circ \ar@/^0.3pc/@<0.7ex>[ur]^{1} \ar@/^-0.3pc/@<0.2ex>[rr]^{2} & &
         \circ
         \ar@/^0.3pc/@<0.7ex>[ll]^{2}
         }$
         \end{minipage}
     &  \begin{minipage}{0.6in}$\xymatrix@C=3mm@R=6mm{
                & \bullet \ar[dr]^{1}             \\
         \circ \ar[ur]^{1} \ar@/^0.3pc/[rr]^{2} & &  \circ  \ar@/^0.3pc/[ll]^{2}
         }$
         \end{minipage}
  & \begin{minipage}{0.6in}
             $\xymatrix@C=3mm@R=6mm{
                & \bullet \ar@/^-0.3pc/@<0.2ex>[dl]^{1}             \\
         \circ \ar@/^0.3pc/@<0.7ex>[ur]^{1} \ar@/^-0.3pc/@<0.2ex>[rr]^{2} & & *+[o][F-]{1}
         \ar@/^0.3pc/@<0.7ex>[ll]^{1}
         }$
         \end{minipage}
     &  \begin{minipage}{0.6in}
             $\xymatrix@C=3mm@R=6mm{
                & \bullet \ar@/^-0.3pc/@<0.2ex>[dl]^{1}             \\
         *+[o][F-]{1} \ar@/^0.3pc/@<0.7ex>[ur]^{1} \ar@/^-0.3pc/@<0.2ex>[rr]^{1} & & *+[o][F-]{1}
         \ar@/^0.3pc/@<0.7ex>[ll]^{1}
         }$
         \end{minipage}
     &  \begin{minipage}{0.6in}$\xymatrix@C=3mm@R=6mm{
                & \bullet \ar[dr]^{1}             \\
         *+[o][F-]{1} \ar[ur]^{1} \ar@/^0.3pc/[rr]^{1} & &  *+[o][F-]{1} \ar@/^0.3pc/[ll]^{1}
         }$
         \end{minipage}
     &   \begin{minipage}{0.6in} \xymatrix{
            \bullet   \ar@/^0.3pc/[r]^1
                & \circ \ar@/^0.3pc/[l]^1 \ar@/^0.3pc/[d]^{1}  \\
            *+[o][F-]{1} \ar@/^0.3pc/[r]^1
                & \circ \ar@/^0.3pc/[l]^1 \ar@/^0.3pc/[u]^{1}           }
                \end{minipage}
\\
\hline $-3/4$ & $-3/4$ & $-1$ & $-1$ & $-1$ & $1$
\\
\hline
\begin{minipage}{0.7in} \xymatrix{
            \bullet   \ar@/^0.3pc/[r]^1
                & \circ \ar@/^0.3pc/[l]^1 \ar[dl]|-{1} \\
            \circ \ar@/^0.3pc/[r]^2
                & \circ \ar@/^0.3pc/[l]^1 \ar[u]_{1}           }
                \end{minipage}
     &  \begin{minipage}{0.7in}\xymatrix{
  \bullet  \ar[d]^{1}
                & \circ \ar[l]_{1} \ar@/^0.3pc/[d]^{1}  \\
  \circ  \ar@{>}[ur]|-{1} \ar@/^0.3pc/[r]^1
                & \circ \ar@/^0.3pc/[u]^{1}   \ar@/^0.3pc/[l]^1         }
                \end{minipage}
  & \begin{minipage}{0.7in} \xymatrix{
            \bullet   \ar@/^0.3pc/[r]^1
                & \circ \ar@/^0.3pc/[l]^1  \ar[dl]|-{1} \\
            *+[o][F-]{1} \ar[r]_1
                & *+[o][F-]{1} \ar[u]_{1}           }
                \end{minipage}
     &  \begin{minipage}{0.7in}
\xymatrix{
  \bullet \ar[d]_{1}
                & \circ  \ar[l]_{1} \ar@/^0.3pc/@{>}[dl]^{1}\\
  \circ \ar@/^0.3pc/@{>}[ur]^{1} \ar[r]_{1}
                & *+[o][F-]{1}   \ar[u]_{1}         }
\end{minipage}
     &  \begin{minipage}{0.7in}
\xymatrix{
  \bullet \ar[d]_{1}
                & \circ  \ar[l]_{1} \ar[d]^{1}\\
  \circ \ar[ur]^{2}
                & *+[o][F-]{1}   \ar[l]^{1}         }
\end{minipage}
     &   \begin{minipage}{0.7in}
\xymatrix{
  \bullet \ar[d]_{1}
                & \circ \ar[dl]_{1} \ar[l]_{1} \\
  \circ  \ar[r]_{2}
                & \circ   \ar[u]_{2}         }
\end{minipage}
\\
\hline $3/2$ & $2$ & $1$ & $1$ & $1$ & $3/4$
\\
\hline
\end{tabular}
\end{table}
\begin{table}[h] \footnotesize
\centering \caption{$\Gamma\in\dot{\mathcal G}_{scon}(4)$ with
$Q_\Gamma=0$} \label{tb3}
\begin{tabular}{|c|c|c|c|c|}

\hline $\xymatrix{
        *+[o][F-]{1} \ar@/^/[r]^1
         &
        \bullet\,2 \ar@/^/[l]^1} $
     & $\xymatrix{
         *+[o][F-]{1} \ar@/^/[r]^1
         &
        \bullet\,1 \ar@/^/[l]^2} $
      & $\xymatrix{
         *+[o][F-]{1} \ar@/^/[r]^1
         &
        \bullet \ar@/^/[l]^3} $
     &  $\xymatrix{
         *+[o][F-]{1} \ar@/^/[r]^2
         &
        \bullet\,1 \ar@/^/[l]^1} $
     & $\xymatrix{
         *+[o][F-]{1} \ar@/^/[r]^2
         &
        \bullet \ar@/^/[l]^2} $
\\
\hline $\xymatrix{
         *+[o][F-]{1} \ar@/^/[r]^3
         &
        \bullet \ar@/^/[l]^1} $
     & \begin{minipage}{0.6in}
             $\xymatrix@C=3mm@R=6mm{
                & \bullet \ar@/^-0.3pc/@<0.2ex>[dl]^{1}\ar[dr]^{1}            \\
         \circ \ar@/^0.3pc/@<0.7ex>[ur]^{2} & & *+[o][F-]{1}
         \ar[ll]^{1}
         }$
         \end{minipage}
     & \begin{minipage}{0.6in}
             $\xymatrix@C=3mm@R=6mm{
                & \bullet \ar[dr]^{1}            \\
         \circ \ar[ur]^{2}  & & *+[o][F-]{1}
         \ar[ll]^{2}
         }$
         \end{minipage}
     &  \begin{minipage}{0.6in}
             $\xymatrix@C=3mm@R=6mm{
                & \bullet \ar@/^-0.3pc/@<0.2ex>[dl]^{2}            \\
         \circ \ar@/^0.3pc/@<0.7ex>[ur]^{1} \ar[rr]_{1} & & *+[o][F-]{1}
         \ar[ul]_{1}
         }$
         \end{minipage}
     &  \begin{minipage}{0.6in}
             $\xymatrix@C=3mm@R=6mm{
                & \bullet \ar@/^-0.3pc/@<0.2ex>[dl]^{1} \ar@/^0.3pc/@<0.2ex>[dr]^{1}          \\
         \circ \ar@/^0.3pc/@<0.7ex>[ur]^{1} \ar@/^-0.3pc/@<0.2ex>[rr]^{1} & &
         \circ
         \ar@/^0.3pc/@<0.7ex>[ll]^{1} \ar@/^-0.3pc/@<0.7ex>[ul]^{1}
         }$
         \end{minipage}
\\
\hline
       \begin{minipage}{0.6in}
             $\xymatrix@C=3mm@R=6mm{
                & \bullet \ar[dr]^{2}            \\
         *+[o][F-]{1} \ar[ur]^{1}  & & \circ
         \ar[ll]^{2}
         }$
         \end{minipage}
  & \begin{minipage}{0.7in}
             $\xymatrix@C=3mm@R=6mm{
                & \bullet \ar@/^-0.3pc/@<0.2ex>[dl]^{1} \ar@/^0.3pc/@<0.7ex>[dr]^{1}             \\
         *+[o][F-]{1}  \ar@/^0.3pc/@<0.7ex>[ur]^{1}  & &  *+[o][F-]{1}
         \ar@/^-0.3pc/@<0.2ex>[ul]^{1}
         }$
         \end{minipage}
     & \begin{minipage}{0.6in}$\xymatrix@C=2mm@R=6mm{
                & \bullet\, 1 \ar[dr]^{1}             \\
         *+[o][F-]{1} \ar[ur]^{1}  & &  *+[o][F-]{1} \ar[ll]^{1}
         }$
         \end{minipage}
     &  \begin{minipage}{0.6in}$\xymatrix@C=3mm@R=6mm{
                & \bullet \ar[dr]^{2}             \\
         *+[o][F-]{1} \ar[ur]^{1}  & &  *+[o][F-]{1} \ar[ll]^{1}
         }$
         \end{minipage}
     & \begin{minipage}{0.6in}
             $\xymatrix@C=3mm@R=6mm{
                & \bullet \ar@/^-0.3pc/@<0.2ex>[dl]^{1}  \ar[dr]^{1}           \\
         *+[o][F-]{1} \ar@/^0.3pc/@<0.7ex>[ur]^{1}  & & *+[o][F-]{1}
         \ar[ll]^{1}
         }$
         \end{minipage}
\\
\hline
     \begin{minipage}{0.6in}
             $\xymatrix@C=3mm@R=6mm{
                & \bullet \ar@/^-0.3pc/@<0.2ex>[dl]^{1}            \\
         *+[o][F-]{1} \ar@/^0.3pc/@<0.7ex>[ur]^{1} \ar[rr]_{1} & & *+[o][F-]{1}
         \ar[ul]_{1}
         }$
         \end{minipage}
  &  \begin{minipage}{0.6in}$\xymatrix@C=3mm@R=6mm{
                & \bullet \ar[dr]^{1}             \\
        *+[o][F-]{1} \ar[ur]^{1}  & &  *+[o][F-]{2} \ar[ll]^{1}
         }$
         \end{minipage}
     &  \begin{minipage}{0.6in}$\xymatrix@C=3mm@R=6mm{
                & \bullet \ar[dr]^{1}             \\
        *+[o][F-]{1} \ar[ur]^{1}  & &  *+[o][F-]{1} \ar[ll]^{2}
         }$
         \end{minipage}
     &  \begin{minipage}{0.6in}$\xymatrix@C=3mm@R=6mm{
                & \bullet \ar[dr]^{1}             \\
        *+[o][F-]{1} \ar[ur]^{2}  & &  *+[o][F-]{1} \ar[ll]^{1}
         }$
         \end{minipage}
   & \begin{minipage}{0.6in}$\xymatrix@C=3mm@R=6mm{
                & \bullet \ar[dr]^{1}             \\
        *+[o][F-]{2} \ar[ur]^{1}  & &  *+[o][F-]{1} \ar[ll]^{1}
         }$
         \end{minipage}
\\
\hline
        \begin{minipage}{0.7in}\xymatrix{
  \bullet  \ar[d]_{1}
                & \circ \ar[l]_{1} \ar@/^0.3pc/[d]^{1}  \\
*+[o][F-]{1}   \ar[r]_1
                & \circ \ar@/^0.3pc/[u]^{2}          }
                \end{minipage}
  & \begin{minipage}{0.7in} \xymatrix{
            \bullet   \ar[r]^1
                & *+[o][F-]{1}  \ar[d]^{1} \\
            *+[o][F-]{1}  \ar@/^0.3pc/[r]^1
                & \circ \ar@/^0.3pc/[l]^1 \ar[ul]_{1}           }
                \end{minipage}
     &  \begin{minipage}{0.7in}\xymatrix{
  \bullet  \ar[r]^{1}
                & \circ  \ar@/^0.3pc/[d]^{2}  \\
*+[o][F-]{1}   \ar[u]^1
                & \circ \ar@/^0.3pc/[u]^{1}    \ar[l]^1       }
                \end{minipage}
     &  \begin{minipage}{0.7in} \xymatrix{
            \bullet   \ar[dr]^{1}
                & *+[o][F-]{1}  \ar[l]_{1} \\
            *+[o][F-]{1}  \ar@/^0.3pc/[r]^1
                & \circ \ar@/^0.3pc/[l]^1   \ar[u]_{1}        }
                \end{minipage}
     & \begin{minipage}{0.7in}\xymatrix{
  \bullet  \ar[d]_{1}
                & *+[o][F-]{1} \ar[l]_{1}  \\
*+[o][F-]{1}   \ar[r]_1
                & *+[o][F-]{1}    \ar[u]_1       }
                \end{minipage}
\\
\hline
\end{tabular}
\end{table}

$$ \ \ \ \ $$

\

\end{document}